\documentclass[12pt,a4paper,reqno]{amsart}
\usepackage{amssymb}
\usepackage{amscd}
\usepackage{mathtools}
\numberwithin{equation}{section}

\theoremstyle{plain}

\newtheorem{theorem}[subsection]{Theorem}
\newtheorem{proposition}[subsection]{Proposition}
\newtheorem{lemma}[subsection]{Lemma}
\newtheorem{corollary}[subsection]{Corollary}

\theoremstyle{definition}

\newtheorem{definition}[subsection]{Definition}

\newtheorem{remark}[subsection]{Remark}

\newcommand\m{\mathrm{m}}
\newcommand\n{{\mathrm{n}}}
\newcommand\h{\mathrm{h}}

\newcommand\E{\mathbb{E}}
\newcommand\Z{\mathbb{Z}}
\newcommand\R{\mathbb{R}}

\newcommand\C{\mathbb{C}}
\newcommand\N{\mathbb{N}}

\newcommand\eps{\varepsilon}

\renewcommand\mod{{\ \operatorname{mod}\ }}

\begin{document}

\title[Polynomial patterns in the primes]{Polynomial patterns in the primes}

\author{Terence Tao}
\address{UCLA Department of Mathematics, Los Angeles, CA 90095-1596.
}
\email{tao@math.ucla.edu}

\author{Tamar Ziegler}
\address{Einstein Institute of Mathematics,
Edmond J. Safra Campus, Givat Ram 
The Hebrew University of Jerusalem,
Jerusalem, 91904, Israel}
\email{tamarz@math.huji.ac.il}

\thanks{The first author is supported by NSF grant DMS-1266164 and by a Simons Investigator Award. The second author is supported by ISF grant 407/12.}

\subjclass{11B30, 37A15}

\begin{abstract}  Let $P_1,\dots,P_k \colon \Z \to \Z$ be polynomials of degree at most $d$ for some $d \geq 1$, with the degree $d$ coefficients all distinct, and admissible in the sense that for every prime $p$, there exists integers $n,m$ such that $n+P_1(m),\dots,n+P_k(m)$ are all not divisible by $p$.  We show that there exist infinitely many natural numbers $n,m$ such that $n+P_1(m),\dots,n+P_k(m)$ are simultaneously prime, generalizing a previous result of the authors, which was restricted to the special case $P_1(0)=\dots=P_k(0)=0$ (though it allowed for the top degree coefficients to coincide).   Furthermore, we obtain an asymptotic for the number of such prime pairs $n,m$ with $n \leq N$ and $m \leq M$ with $M$ slightly less than $N^{1/d}$.

Our arguments rely on four ingredients.  The first is a (slightly modified) generalized von Neumann theorem of the authors, reducing matters to controlling certain averaged local Gowers norms of (suitable normalizations of) the von Mangoldt function.  The second is a more recent concatenation theorem of the authors, controlling these averaged local Gowers norms by global Gowers norms.  The third ingredient is the work of Green and the authors on linear equations in primes, allowing one to compute these global Gowers norms for the normalized von Mangoldt functions.  Finally, we use the Conlon-Fox-Zhao densification approach to the transference principle to combine the preceding three ingredients together.

In the special case $P_1(0)=\dots=P_k(0)=0$, our methods also give infinitely many $n,m$ with $n+P_1(m),\dots,n+P_k(m)$ in a specified set primes of positive relative density $\delta$, with $m$ bounded by $\log^L n$ for some $L$ independent of the density $\delta$.  This improves slightly on a result from our previous paper, in which $L$ was allowed to depend on $\delta$.
\end{abstract}

\maketitle

\section{Introduction}

In 2008, Green and the first author \cite{gt-primes} established that the primes contain arbitrarily long arithmetic progressions.  Equivalently, one has

\begin{theorem}\label{gt}  Let $k$ be a natural number (i.e. an element of $\N = \{1,2,3,\dots\}$).  Then there exist infinitely many natural numbers $n,m$ such that $n, n+m, \dots,n+(k-1)m$ are all prime.
\end{theorem}

Among the ingredients of the proof of \cite{gt-primes} was the deployment of the (global) Gowers uniformity norms introduced in \cite{gow-szem-4}, \cite{gow-szem}, the development of a \emph{generalized von Neumann theorem} controlling certain multiple averages by these norms, and Szemer\'edi's theorem \cite{szem} on arithmetic progressions.

Theorem \ref{gt} was generalized by the authors \cite{tz-pattern} (see also \cite{hoang-wolf} for a subsequent refinement):

\begin{theorem}\label{tz}  Let $P_1,\dots,P_k \in \Z[\m]$ be polynomials in one indeterminate $\m$ such that $P_1(0)=\dots=P_k(0)=0$.  Then there exist infinitely many natural numbers $n,m$ such that $n+P_1(m), n+P_2(m), \dots,n+P_k(m)$ are all prime.
\end{theorem}

The proof of Theorem \ref{tz} broadly followed the arguments used to prove Theorem \ref{gt}; for instance, a generalized von Neumann theorem continued to play a crucial role.  However, there are some key differences between the two arguments.  Most notably, the global Gowers uniformity norms used in \cite{gt-primes} were replaced by more complicated \emph{averaged local Gowers uniformity norms}, and the Szemer\'edi theorem was replaced with the multidimensional polynomial recurrence theorem of Bergelson and Leibman \cite{bl}.  It was necessary to deploy the multidimensional version of this theorem (despite the one-dimensional nature of Theorem \ref{tz}) in order to obtain some uniformity in the recurrence theorem with respect to a certain technical parameter $W$ that arose in the proof.  In \cite{tz-pattern}, the natural numbers $n,m$ could be chosen so that $m = O(n^{o(1)})$ as $n \to \infty$; in \cite{tz-narrow} this was improved to $m = O(\log^{O(1)} n)$ (with the implied constants depending on $P_1,\dots,P_k$).  In \cite{le-wolf} it was also shown that $m$ could be taken to be one less than a prime.

In a series of papers \cite{gt-linear}, \cite{gt-mobius}, \cite{gtz-uk}, a different generalization of Theorem \ref{gt} was obtained, namely that the arithmetic progression $n,n+r,\dots,n+(k-1)r$ was replaced by a more general pattern, and a more quantitative count of the prime patterns was established.   We give a special case of the main results of that paper as follows.  If $A$ is a finite non-empty set, we use $|A|$ to denote its cardinality, and for any function $f \colon A \to \C$, we write $\E_{a \in A} f(a)$ for $\frac{1}{|A|} \sum_{a \in A} f(a)$.  For any $N$, we let $[N]$ denote the discrete interval
$$[N] \coloneqq  \{ n \in \N: n \leq N \},$$
and let $\Lambda \colon \Z \to \R$ denote the von Mangoldt function (thus $\Lambda(p^j) \coloneqq  \log p$ for all primes $p$ and natural numbers $j$, with $\Lambda(n)=0$ otherwise).  All sums and products over $p$ are understood to be restricted to the primes unless otherwise specified.  We will be using $O()$ and $o()$ asymptotic notation; we review our conventions for this in Section \ref{notation-sec} below.

\begin{theorem}\label{linear-pattern}  Let $P_1,\dots,P_k \in \Z[\m]$ be linear polynomials (thus $P_i(\m) = a_i \m + b_i$ for some integers $a_i,b_i$, for each $i=1,\dots,k$).  Assume that the leading coefficients $a_1,\dots,a_k$ of the polynomials $P_1,\dots,P_k$ are distinct.  Let $N$ be an asymptotic parameter going to infinity, and let $M = M(N)$ be a quantity such that $M = o_{N \to \infty}( N)$ and $M \gg N \log^{-A} N$ for some fixed constant $A$.  Then
$$
\E_{n \in [N]} \E_{m \in [M]} \Lambda(n+P_1(m)) \dots \Lambda(n+P_k(m)) = \prod_p \beta_p + o_{N \to \infty}(1) $$
where for each prime $p$, $\beta_p$ is the local factor
$$ \beta_p \coloneqq  \E_{n \in \Z/p\Z} \E_{m \in \Z/p\Z} \Lambda_p(n+P_1(m)) \dots \Lambda_p(n+P_k(m))$$
and $\Lambda_p \colon \Z/p\Z \to \R$ is the function $\Lambda_p(n) \coloneqq  \frac{p}{p-1} 1_{n \neq 0 \mod  p}$.
\end{theorem}

We remark that it is easy to establish the absolute convergence of the product $\prod_p \beta_p$; see \cite{gt-linear}.

\begin{proof}  If $N/M$ is a sufficiently slowly growing function of $N$, this follows immediately from the main theorem of \cite{gt-linear} (specialized to the finite complexity tuple $\n + P_1(\m),\dots,\n+P_k(\m)$ of linear forms), together with the results of \cite{gt-mobius}, \cite{gtz-uk} proving the two conjectures assumed in \cite{gt-linear}.  The extension to the case where $N/M$ is allowed to grow as fast as a power of $\log N$ follows from the same arguments, as discussed in \cite[Appendix A]{fgkt}.
\end{proof}

Specializing this theorem to the case $P_i(\m) \coloneqq  (k-1)\m$, one can compute that $\prod_p \beta_p$ is non-zero, and then it is an easy matter to see that Theorem \ref{linear-pattern} implies Theorem \ref{gt}.  The Hardy-Littlewood prime tuples conjecture \cite{hardy} predicts that the condition that the leading coefficients of the $P_1,\dots,P_k$ are distinct can be relaxed to the condition that the $P_1,\dots,P_k$ themselves are distinct, which would imply (among other things) the twin prime conjecture, but this is unfortunately well beyond the known techniques used to prove this theorem.  The arguments in \cite{gt-linear} also treat the case when $M/N$ is comparable to one, as long as an Archimedean local factor $\beta_\infty$ is inserted on the right-hand side, but to simplify the arguments slightly we work in the local setting $M = o_{N \to\infty}(N)$ to avoid having to consider the Archimedean factor.  The prime tuples conjecture also predicts that the condition $M \gg N \log^{-A} N$ can be replaced with the much weaker condition that $M$ goes to infinity, but again this is beyond the reach of current methods. Even the special case $\E_{n \in [N]} \E_{m \in [M]} \Lambda(n) \Lambda(n+m) = 1 + o_{N \to \infty}(1)$, which essentially measures the error term for the prime number theorem in short intervals on the average, is only known\footnote{If one replaces the von Mangoldt function with the M\"obius function, then there is recent work \cite{mr}, \cite{mrt} obtaining such asymptotics for $H$ growing arbitrarily slowly with $N$.  However, the techniques used rely heavily on small prime divisors, and so do not seem to be directly applicable to problems involving the von Mangoldt function.} for $M \geq N^{1/6+\eps}$ (using the zero-density estimates of Huxley \cite{huxley}), or for $M \geq N^\eps$ assuming the Riemann hypothesis, while to the best of our knowledge the expected asymptotic for $\E_{n \in [N]} \E_{m \in [M]} \Lambda(n) \Lambda(n+m) \Lambda(n+2m)$ is only known for $M \geq N^{5/8+\eps}$ (using the exponential sum estimates of Zhan \cite{zhan}).

The proof of Theorem \ref{linear-pattern} used the same global Gowers uniformity norms that appeared in the proof of Theorem \ref{gt}, as well as a very similar generalized von Neumann theorem.  However, Szemer\'edi's theorem was no longer used, as this result does not hold for arbitrary linear polynomials $P_1,\dots,P_k$ and in any event only provides lower bounds for multiple averages, as opposed to asymptotics.  Instead, by using the results of \cite{gt-mobius}, \cite{gtz-uk} together with some transference arguments, it was shown that a suitable normalization $\Lambda'_{b,W}-1$ of the von Mangoldt function $\Lambda$ was small with respect to the global Gowers uniformity norm, which is sufficient to establish the stated result thanks to the generalized von Neumann theorem.

The first main result of this paper is a higher degree generalization of Theorem \ref{linear-pattern}, which (except for a technical additional condition regarding the lower bound on $M$) is to Theorem \ref{tz} as Theorem \ref{linear-pattern} is to Theorem \ref{gt}.  More precisely, we show

\begin{theorem}[Main theorem]\label{main}  Let $d, r$ be natural numbers, and let $P_1,\dots,P_k \in \Z[\m_1,\dots,\m_r]$ be polynomials of integer coefficients of degree at most $d$.  Furthermore, assume that the degree $d$ components of $P_1,\dots,P_k$ are all distinct (or equivalently, that $P_i-P_j$ has degree exactly $d$ for all $1 \leq i < j \leq k$).  Let $N$ be an asymptotic parameter going to infinity, and let $M = M(N)$ be a quantity such that $M/N^{1/d}$ goes to zero sufficiently slowly as $N \to \infty$ (that is to say, there is a quantity $\omega(N)$ going to zero as $N \to \infty$ depending only on $d,r,k,P_1,\dots,P_k$, and we assume that $M = o_{N \to \infty}( N^{1/d})$ and $M \geq \omega(N) N^{1/d}$).  Then
\begin{equation}\label{lal}
\E_{n \in [N]} \E_{\vec m \in [M]^r} \Lambda(n+P_1(\vec m)) \dots \Lambda(n+P_k(\vec m)) = \prod_p \beta_p + o_{N \to \infty}(1) 
\end{equation}
where $\beta_p$ is the local factor
\begin{equation}\label{betap}
\beta_p \coloneqq  \E_{n \in \Z/p\Z} \E_{\vec m \in (\Z/p\Z)^r} \Lambda_p(n+P_1(\vec m)) \dots \Lambda_p(n+P_k(\vec m))
\end{equation}
and $\Lambda_p \colon \Z/p\Z \to \R$ is the function $\Lambda_p(n) \coloneqq  \frac{p}{p-1} 1_{n \neq 0 \mod p}$.
\end{theorem}

Note from \cite[Lemma 9.5]{tz-pattern} gives the asymptotic $\beta_p = 1 + O_{P_1,\dots,P_k}(\frac{1}{p^2})$ for all sufficiently large $p$ (noting that all such $p$ are ``good'' in the sense of \cite[Definition 9.4]{tz-pattern}), so the product $\prod_p \beta_p$ is convergent.

It is likely that the lower bound $M \geq\omega(N) N^{1/d}$ can be relaxed to $M \geq  N^{1/d}\log^{-A} N$ as in Theorem \ref{linear-pattern}, and the upper bound relaxed to $M=O(N)$ at the cost of inserting an Archimedean factor $\beta_\infty$ on the right-hand side as in \cite{gt-linear}, but we do not attempt to establish these extensions here to simplify the exposition.  As will be clear from the method of proof, one can allow for much smaller values of $M$ - in principle, as small as $\log^L N$ for some large $L$ - as soon as one is able to establish some local Gowers uniformity for (a ``$W$-tricked'' modification of) the von Mangoldt function at scale $M^d$.

As with Theorem \ref{linear-pattern}, standard conjectures such as the Bateman-Horn conjecture \cite{bateman} predict that Theorem \ref{main} continues to hold without the requirement that the degree $d$ components of $P_i$ are distinct (so long as the $P_i$ themselves remain distinct), and with $M$ growing arbitrarily slowly with $N$.  Such strengthenings of Theorem \ref{main} remain beyond the methods here.  We also remark that a result similar to Theorem \ref{main}, involving more general polynomial patterns, was established in \cite{cook} under an assumption of sufficiently large ``Birch rank'' on the polynomial pattern, which is a rather different regime to the one considered here in as it tends to require a large number of variables compared to the number and degree of polynomials.  In the recent paper \cite{bien}, some special cases of \eqref{main} were established, in particular the case when $r=2$ and $P_i(\m_1,\m_2) = i (\m_1^2+\m_2^2)$ for $i=1,\dots,k$.

As in \cite{gt-linear}, we have a qualitative corollary of Theorem \ref{main}	:

\begin{corollary}[Qualitative main theorem]\label{qmt} Let $d,r$ be natural numbers, $P_1,\dots,P_k \in \Z[\m_1,\dots,\m_r]$ be polynomials of degree at most $d$.  Assume that the degree $d$ components of $P_1,\dots,P_k$ are all distinct, and suppose that for each prime $p$ there exist $n \in \Z$ and $\vec m \in \Z^r$ such that $n+P_1(\vec m),\dots,n+P_k(\vec m)$ are all not divisible by $p$.  Then there exist infinitely many natural numbers $n,m_1,\dots,m_r$ such that $n+P_1(m_1,\dots,m_r),\dots,n+P_k(m_1,\dots,m_r)$ are simultaneously prime.
\end{corollary}

\begin{proof}  (Sketch) From hypothesis, the local factors $\beta_p$ are all non-zero, and one can establish the asymptotic $\beta_p = 1 + O_{k,d,r}(\frac{1}{p^2})$ for all sufficiently large primes $p$.  We conclude that the Euler product $\prod_p \beta_p$ is non-zero, and the claim then follows from Theorem \ref{main} since the contribution when one of the $n+P_i(m_1,\dots,m_r)$ is a prime power, rather than a prime, can easily be shown to be negligible.
\end{proof}

The case $r=1, k=2$ of this corollary had been previously established in \cite{le}.
Corollary \ref{qmt} is a partial generalisation of Theorem \ref{tz}, in that it implies the special case of that theorem when $r=1$ and the $\m^d$ coefficients of the polynomials $P_1(\m),\dots,P_k(\m)$ are all distinct.   As mentioned above, the arguments here should eventually extend to allow the $\m^d$ coefficients of $P_i$ to be distinct, so long as the $P_i - P_j$ are non-constant, once one can establish local Gowers norm control on the von Mangoldt function.  In fact, Schinzel's hypothesis H \cite{schinzel} predicts that Corollary \ref{qmt} shold hold even if the some of the $P_i - P_j$ are constant, but this claim (which includes the twin prime conjecture as a special case) is well beyond the methods of this paper.  

We now briefly summarize the method of proof of Theorem \ref{main}.  If one directly applies the methods used to prove Theorem \ref{linear-pattern}, replacing (a variant of) the generalized von Neumann theorem from \cite{gt-primes} with (a variant of) the more complicated generalized von Neumann theorem from \cite{tz-pattern}, one ends up wishing to control various normalized versions $\Lambda'_{b,W}-1$ of the von Mangoldt function in certain averaged local Gowers uniformity norms; furthermore, the transference machinery in \cite{gt-primes}, \cite{tz-pattern} allows one to replace $\Lambda_{b,W}-1$ by a bounded function.  At this point, one would like to apply an inverse theorem for the averaged local Gowers uniformity norms, but a direct application of the known inverse theorems does not lead to a particularly tractable condition to verify on $\Lambda_{b,W}-1$.  To overcome this issue, we use the concatenation theorems recently developed by us in \cite{tz-concat} to control the averaged local Gowers uniform norms arising from the generalised von Neumann theorem by a global Gowers uniform norm.  It is at this juncture that the hypothesis that the $\m^d$ coefficients of the $P_i(\m)$ are all distinct, together with the choice of $M$ as being close to $N^{1/d}$, becomes crucial. 

In some cases, we are able use our methods to partially remove the requirement in Theorem \ref{main} that the degree $d$ components of $P_1,\dots,P_k$ are distinct, although this requires one to understand the distribution (or Gowers uniformity) of primes in short intervals, for which the known unconditional results still fall well short of what is conjecturally true.  As an example of this, we give the conjectural asymptotic for prime triplets of the form $n, n+m, n+P(m)$, with $n \in [N]$ and $m$ slightly smaller than $N^{1/d}$, when $P$ has degree exactly $d$:

\begin{theorem}\label{nn2}  Theorem \ref{main} is true in the case $r=1$, $k=3$, $P_1 = 0$, $P_2 = \m$, and when $P_3$ has degree exactly $d$ for some $2 \leq d \leq 5$.  If one assumes the generalized Riemann hypothesis (GRH), the condition $d \leq 5$ may be removed.
\end{theorem}

This theorem will be proven at the end of Section \ref{w-trick}.  The main idea is to use the machinery of proof of Theorem \ref{main} to essentially eliminate the $\Lambda(n+P_3(m))$ factor, leaving only the task of controlling averages roughly of the form $\E_{n \in [N]} \E_{m \in [M]} \Lambda(n) \Lambda(n+m)$, which can be handled by existing results on primes in short intervals, both with and without GRH.  We briefly discuss some other cases that can be handled by this method at the end of that section.

The above results have (somewhat simpler) analogues when the von Mangoldt function $\Lambda$ is replaced by the M\"obius function $\mu$ (or the closely related Liouville function $\lambda$).  In these analogues, the local factors $\prod_p \beta_p$ should be deleted, thus for instance we have
$$\E_{n \in [N]} \E_{m \in [M]^r} \mu(n+P_1(m)) \dots \mu(n+P_k(m)) = o_{N \to \infty}(1) $$
under the hypotheses of Theorem \ref{main}.  The proof of these variant results is in fact significantly simpler, as all the pseudorandom measures $\nu$ that appear in the arguments below can be simply replaced by $1$; also, the ``$W$-trick'' is not needed in this case.  We leave the modifications of the arguments below needed to obtain these variants to the interested reader.

The methods used to establish Corollary \ref{qmt} also give a variant involving sets of primes of positive upper density, improving slightly on our previous results in \cite{tz-narrow}.

\begin{theorem}[Narrow polynomial patterns in subsets of the primes]\label{main-sub}  Let $d$ be a natural number, and let $P_1,\dots,P_k \in \Z[\m]$ be polynomials of integer coefficients of degree at most $d$, such that $P_1(0)=\dots=P_k(0)=0$.  Let $L$ be a sufficiently large quantity depending on $d$.  Let $\delta > 0$, and let ${\mathcal A}$ be a subset of the primes ${\mathcal P}$ such that
$$ \limsup_{N \to \infty} \frac{|{\mathcal A} \cap [1,N]|}{|{\mathcal P} \cap [1,N]} \geq \delta$$
for some $\delta > 0$.  Then one can find infinitely many natural numbers $n,m$ with $n+P_1(m),\dots,n+P_k(m)$ in ${\mathcal A}$ and with $m \leq \log^L n$, where $L$ depends only on $d,k$.
\end{theorem}

Comparing this result with Corollary \ref{qmt}, we see that the set of primes ${\mathcal P}$ has been replaced by a positive density subset ${\mathcal A}$, and that the condition that the $P_i-P_j$ have degree exactly $d$ has been dropped, replaced instead by the hypothesis $P_1(0)=\dots=P_k(0)$; we have also set $r=1$ (as the $r>1$ case follows easily from the $r=1$ case, e.g. by restricting to the diagonal $m_1=\dots=m_r$).  Another key point is the smallness condition on $m$.  The arguments in \cite{tz-pattern} essentially established this theorem with the bound $m = n^{o(1)}$, and the subsequent argument in \cite{tz-narrow} improved this to $m \leq \log^{L(\delta)} n$ where the exponent $L(\delta)$ is permitted to depend on $\delta$ in addition to $d$ and $k$.  Thus the new contribution of Theorem \ref{main-sub} is the removal of the dependence of $L$ on $\delta$.  This follows from the arguments in \cite{tz-narrow}, after replacing \cite[Theorem 9]{tz-narrow} with Theorem \ref{l1} below; see Remark \ref{ms}.

\subsection{Notation}\label{notation-sec}

If $x \in \R/\Z$, we use $\|x\|_{\R/\Z}$ to denote the distance from $x$ to the nearest integer, and $e(x) \coloneqq  e^{2\pi i x}$.

Given two real numbers $A,B$ with $A \leq B$, we use $[A,B]$ to denote the discrete interval $\{ n \in \Z: A \leq n \leq B \}$.  We will often need to identify this interval with a subset of the cyclic group $\Z/N\Z$ for some modulus $N$ (larger than $B-A$), which can of course be done by applying the reduction map $n \mapsto n\ (N)$ from $\Z$ to $\Z/N\Z$.  Similarly for the interval $[M] \coloneqq  \{n \in \N: n \leq M \}$.

Given any finite collection $\m_1,\dots,\m_r$ of indeterminates, we write $\Z[\m_1,\dots,\m_r]$ for the ring of formal polynomials $P$ in these variables with integer coefficients, thus there is a natural number $d$ for which one has
$$ P = \sum_{i_1,\dots,i_r \geq 0: i_1+\dots+i_r \leq d} \alpha_{i_1,\dots,i_r} \m_1^{i_1} \dots \m_r^{i_r}$$
for some integers $\alpha_{i_1,\dots,i_r}$.  The least such $d$ for which one has such a representation is the \emph{degree} of $P$.  Of course, one can evaluate $P(m_1,\dots,m_r)$ for any elements $m_1,\dots,m_r$ of a commutative ring (such as $\Z$ or $\Z/q\Z$) simply by substituting the indeterminates $\m_i$ with their evaluations $m_i$.  We will write indeterminate variables such as $\m,\n$ in Roman font, to distinguish them from elements $m,n$ of a specific ring such as $\Z$ or $\Z/q\Z$.

It will be convenient to work with the notion of a \emph{finite multiset} - an unordered collection $\{a_1,\dots,a_n\}$ of a finite number of objects $a_1,\dots,a_n$, in which repetitions are allowed.  This clearly generalises the notion of a finite set, in which every element occurs with multiplicity one.  If $A = \{a_1,\dots,a_n\}$ is a non-empty finite multiset, and $f \colon A \to \C$ is a function on the elements of $A$, we write
$$ \E_{a \in A} f(a) \coloneqq  \frac{1}{n} \sum_{i=1}^n f(a_i),$$
which of course generalises the notion of an average $\E_{a \in A} f(a)$ over a finite non-empty set $A$; thus for instance $\E_{a \in \{1,2,3,3\}} a = \frac{1+2+3+3}{4}$.  If $A = \{a_1,\dots,a_n\}$ and $B = \{b_1,\dots,b_m\}$ are finite multisets taking values in an additive group $G = (G,+)$, we write
$$ A+B \coloneqq  \{ a_i + b_j: i=1,\dots,n; j=1,\dots,m\}$$
and
$$ A-B \coloneqq  \{ a_i - b_j: i=1,\dots,n; j=1,\dots,m\}$$
for the sum and difference multisets.

Given a finite non-empty multiset $A$ in a domain $X$, we define the \emph{density function} $p_A \colon X \to [0,1]$ to be the quantity
$$p_A(x) \coloneqq  \frac{|\{ a \in A: a = x \}|}{|A|}$$
where the numerator is the multiplicity of $x$ in $A$; thus $\sum_{x \in X} p_A(x)=1$, and $|A|$ the cardinality of $A$ counting multiplicity; more generally, we define
$$ \E_{a \in A} f(a) \coloneqq \sum_{x \in X} f(x) p_A(x)$$
for all $f \colon X \to \R$.  We then define the \emph{total variation distance} $d_{TV}(A,B)$ between two finite non-empty multisets $A,B$ in $X$ to be
$$ d_{TV}(A,B) \coloneqq  \sum_{x \in X} |p_A(x)-p_B(x)|.$$
Thus we have
\begin{equation}\label{say}
 |\E_{a \in A} f(a) - \E_{b \in B} f(b)| \leq d_{TV}(A,B) 
\end{equation}
whenever $f \colon X \to \R$ is bounded in magnitude by one.  Informally, if $d_{TV}(A,A+B)$ is small, one can think of $A$ as having an approximate translation invariance with respect to shifts by $B$.  We observe from the triangle inequality and translation invariance that one has the contraction property
\begin{equation}\label{dtv}
 d_{TV}(A+C, B+C) \leq \E_{c \in C} d_{TV}(A+c,B+c) = d_{TV}(A,B)
\end{equation}
for any finite non-empty multisets $A,B,C$.  In particular, for finite non-empty multisets $A,B,C$ one has
\begin{align*}
d_{TV}(A,A+C) &\leq d_{TV}(A,A+B) + d_{TV}(A+B,A+B+C) + d_{TV}(A+B+C,A+C) \\
&\leq d_{TV}(A,A+B) + d_{TV}(B,B+C) + d_{TV}(A+B,A)
\end{align*}
and thus
\begin{equation}\label{dtv-2}
d_{TV}(A,A+C) \leq 2 d_{TV}(A,A+B) + d_{TV}(B,B+C).
\end{equation}
Informally: if $A$ is approximately $B$-invariant, and $B$ is approximately $C$-invariant, then $A$ is approximately $C$-invariant.

Let $G = (G,+)$ be a finite abelian group, and let $Q_1,\dots,Q_d$ be non-empty finite multisets in $G$ with $d \geq 1$.  Given a function $f \colon G \to \R$, we define the \emph{Gowers box norm} $\|f\|_{\Box^d_{Q_1,\dots,Q_d}(G)} = \|f\|_{\Box^d_{Q_1,\dots,Q_d}}$ to be the non-negative real defined by the formula
$$
\|f\|_{\Box^d_{Q_1,\dots,Q_d}}^{2^d} \coloneqq  \E_{x \in G} \E_{h_1 \in Q_1-Q_1,\dots,h_d \in Q_d-Q_d} \prod_{(\omega_1,\dots,\omega_d) \in \{0,1\}^d} f\left(x + \sum_{i=1}^d \omega_i h_i\right);$$
it is easy to see that the right-hand side is non-negative, so that the $\Box^d_{Q_1,\dots,Q_d}$ norm is well defined.  More generally, given functions $f_\omega \colon G \to \R$ for $\omega \in \{0,1\}^d$, we define the \emph{Gowers inner product}
$$ \langle (f_\omega)_{\omega \in \{0,1\}^d} \rangle_{\Box^d_{Q_1,\dots,Q_d}} \coloneqq 
\E_{x \in G} \E_{h_1 \in Q_1-Q_1,\dots,h_d \in Q_d-Q_d} \prod_{(\omega_1,\dots,\omega_d) \in \{0,1\}^d} f_{\omega_1,\dots,\omega_d}\left(x + \sum_{i=1}^d \omega_i h_i\right) .$$
The Gowers norms can also be defined for complex-valued functions by appropriate insertion of complex conjugation operations, but we will not need to do so here.  We recall the \emph{Cauchy-Schwarz-Gowers inequality}
\begin{equation}\label{csg}
\left| \langle (f_\omega)_{\omega \in \{0,1\}^d} \rangle_{\Box^d_{Q_1,\dots,Q_d}}\right| \leq \prod_{\omega \in \{0,1\}^d} \|f_\omega \|_{\Box^d_{Q_1,\dots,Q_d}}
\end{equation}
(see e.g. \cite[Lemma B.2]{gt-linear}).  Among other things, this implies the monotonicity property
\begin{equation}\label{mono}
\|f\|_{\Box^{d-1}_{Q_1,\dots,Q_{d-1}}} \leq \|f\|_{\Box^{d}_{Q_1,\dots,Q_{d}}}
\end{equation}
for $d > 1$, by setting the $f_\omega$ in \eqref{csg} to be $f$ or $1$ in an appropriate fashion. It also implies the triangle inequality for the $\Box^d_{Q_1,\dots,Q_d}$ norm.

When $Q_1=\dots=Q_d=Q$, we refer to the Gowers box norm $\| \|_{\Box^d_{Q,\dots,Q}(G)}$ as the \emph{Gowers uniformity norm} and abbreviate it as $\| \|_{U^d_Q(G)}$ or simply $\| \|_{U^d_Q}$.  Similarly for the Gowers inner product.  

We observe the identity
$$ 
\|f\|_{\Box^d_{Q_1,\dots,Q_d}}^{2^d} = \E_{x \in G} f(x) {\mathcal D}^d_{Q_1,\dots,Q_d}(f)(x)$$
where the \emph{dual function} ${\mathcal D}^d_{Q_1,\dots,Q_d}(f) \colon G \to \R$ is defined by
$$
{\mathcal D}^d_{Q_1,\dots,Q_d}(f)(x) \coloneqq  \E_{h_1 \in Q_1-Q_1,\dots,h_d \in Q_d-Q_d} \prod_{(\omega_1,\dots,\omega_d) \in \{0,1\}^d \backslash \{0\}^d} f\left(x + \sum_{i=1}^d \omega_i h_i\right).$$
Again, when $Q_1=\dots=Q_d=Q$, we abbreviate ${\mathcal D}^d_{Q,\dots,Q}$ as ${\mathcal D}^d_Q$.  Finally, we recall the $L^p$ norms
$$ \|f\|_{L^p(G)} = \|f\|_{L^p} \coloneqq  \left( \E_{x \in G} |f(x)|^p\right)^{1/p}$$
for $1 \leq p < \infty$, with the usual convention
$$ \|f\|_{L^\infty(G)} = \|f\|_{L^\infty} \coloneqq  \sup_{x \in G} |f(x)|.$$

We now set out the asymptotic notation we will use.  We write $X = O(Y)$, $X \ll Y$, or $Y \gg X$ to denote the estimate $|X| \leq C Y$ for a constant $C$.  Often, we will need the implied $C$ to depend on some parameters, which we will indicate with subscripts, thus for instance $X = O_{d,r}(Y)$ denotes the estimate $|X| \leq C_{d,r} Y$ for some quantity $C_{d,r}$ depending only on $d,r$.  
We will often also be working with an asymptotic parameter such as $N$ going to infinity.  In that setting, some quantities will be held fixed (that is, they will be independent of $N$), while others will be allowed to depend on $N$.  We write $X = o_{N \to \infty}(Y)$, or $X = o(Y)$ for short, if we have $|X| \leq c(N) Y$ for some quantity $c(N)$ depending on $N$ and possibly on other fixed quantities that goes to zero as $N \to \infty$ (holding all fixed quantities constant).

\section{Controlling averaged Gowers norms by global Gowers norms}

A key step in our arguments is establishing that averaged Gowers norms, with shifts depending polynomially on the averaging parameter in a non-degenerate fashion, can be controlled by global Gowers norms.  We begin with a polynomial equidistribution lemma, which roughly speaking asserts that exponential sums with polynomial phases can only be large if the polynomial is ``major arc''.

\begin{lemma}[Polynomial equidistribution]\label{pop}  Let $d,r \geq 1$ be natural numbers, and let $P \in \R[\n_1,\dots,\n_r]$ be a polynomial of degree at most $d$, thus
$$ P(\n_1,\dots,\n_r) = \sum_{i_1,\dots,i_r \geq 0: i_1+\dots+i_r \leq d} \alpha_{i_1,\dots,i_r} \n_1^{i_1} \dots \n_r^{i_r}$$
for some coefficients $\alpha_{i_1,\dots,i_r} \in \R$.
Let $N_1,\dots,N_r \geq 1$ and $0 < \eps \leq 1$ be such that
\begin{equation}\label{jj}
 |\E_{n_1 \in [N_1],\dots,n_r \in [N_r]} e(P(n_1,\dots,n_r))| \geq \eps.
\end{equation}
 Then either one has $N_j \ll_{d,r} \eps^{-O_{d,r}(1)}$ for some $1 \leq j \leq r$, or else there exists a natural number $q \ll_{d,r} \eps^{-O_{d,r}(1)}$ such that
$$ \| q \alpha_{i_1,\dots,i_r} \|_{\R/\Z} \ll_{d,r} \frac{\eps^{-O_{d,r}(1)}}{N_1^{i_1} \dots N_r^{i_r}}$$
for all $i_1,\dots,i_r \geq 0$ with $i_1+\dots+i_r \leq d$.
\end{lemma}

\begin{proof}  In the one-dimensional case $r=1$, this follows from standard Weyl sum estimates (see e.g. \cite[Proposition 4.3]{gt-ratner}).  Now suppose inductively that $r > 1$, and that the claim has already been proven for $r-1$.  For brevity, we allow all implied constants in the asymptotic notation here to depend on $r,d$.

From \eqref{jj} we have
$$ |\E_{n_1 \in [N_1],\dots,n_{r-1} \in [N_{r-1}]} e(P(n_1,\dots,n_r))| \gg \eps$$
for at least $\gg \eps N_r$ values of $n_r \in [N_r]$.  On the other hand, we can write
$$ P(n_1,\dots,n_r) = \sum_{i_1,\dots,i_{r-1} \geq 0: i_1+\dots+i_{r-1} \leq d} \alpha_{i_1,\dots,i_{r-1}}(n_r) n_1^{i_1} \dots n_{r-1}^{i_{r-1}}$$
where for each $i_1,\dots,i_{r-1}$, $\alpha_{i_1,\dots,i_{r-1}} \in \Z[\n_r]$ is the polynomial
$$ \alpha_{i_1,\dots,i_{r-1}}(\n_r)\coloneqq   \sum_{i_r \geq 0: i_1+\dots+i_r \leq d} \alpha_{i_1,\dots,i_r} \n_r^{i_r}.$$
Applying the induction hypothesis, we see that for $\gg \eps N_r$ values of $n_r$, one can find $q \ll \eps^{-O(1)}$ such that
\begin{equation}\label{dosh}
 \|  q \alpha_{i_1,\dots,i_{r-1}}(n_r) \|_{\R/\Z} \ll \frac{\eps^{-O(1)}}{N_1^{i_1} \dots N_{r-1}^{i_{r-1}}}
\end{equation}
for all $i_1,\dots,i_{r-1}$.  At present, $q$ can depend on $n_r$, but by the pigeonhole principle we can pass to a set of $\gg \eps^{O(1)} N_r$ values of $n_r$ and make $q$ independent of $n_r$.  Then, for each $i_1,\dots,i_{r-1}$, we have \eqref{dosh} for $\gg \eps^{O(1)} N_r$ values of $n_r$.  Applying \cite[Lemma 4.5]{gt-ratner}, this implies that either $N_r \ll \eps^{-O(1)}$, or else there exists a natural number $k \ll \eps^{-O(1)}$ (possibly depending on $i_1,\dots,i_{r-1}$) such that
$$
 \| k q \alpha_{i_1,\dots,i_{r-1},i_r} \|_{\R/\Z} \ll \frac{\eps^{-O(1)}}{N_1^{i_1} \dots N_{r-1}^{i_{r-1}} N_r^{i_r}}$$
for all $i_r$.  At present $k$ can depend on $i_1,\dots,i_{r-1}$, but by multiplying together all the $k$ associated to the $O(1)$ possible choices of $(i_1,\dots,i_{r-1})$, we may take $k$ independent of $i_1,\dots,i_{r-1}$.  Replacing $q$ by $kq$, we obtain the claim.
\end{proof}

Our next result asserts (roughly speaking) that on the average, a multidimensional progression
$$ Q(\vec h) \coloneqq  P_1(\vec h_1) [-M,M] + \dots + P_k(\vec h_k) [-M,M]$$
(with $\vec h$ of the order of $M$, and $P_1,\dots,P_k$ polynomials of degree exactly $d-1$) will have an approximate global translation symmetry, in the sense that $Q(\vec h)$ is close on total variation to $Q(\vec h) + q [-A^{-2k} M^d, A^{-2k} M^d]$ for some reasonably small $q$ and $A$, if $k$ is large enough depending on $d,r$.  The equidistribution result from Lemma \ref{pop} will play a key role in the proof of this statement.

\begin{theorem}[Approximate global symmetry]\label{mung}  Let $d,r \geq 1$ be natural numbers, and suppose that $k \geq 1$ is a natural number that is sufficiently large depending on $d,r$.  Let $A$ be a quantity that is sufficiently large depending on $d,r,k$, and for each $j=1,\dots,k$, let $P_j \in \Z[\h_1,\dots,\h_r]$ be a polynomial of degree exactly $d-1$ with coefficients that are integers of magnitude at most $A$.  Let $M$ be a quantity that is sufficiently large depending on $d,r,k,A$.   
For any $\vec h = (\vec h_1,\dots, \vec h_k)$ in $(\Z^r)^k$, let $Q(\vec h)$ be the generalised arithmetic progression
$$ Q(\vec h) \coloneqq  P_1(\vec h_1) [-M,M] + \dots + P_k(\vec h_k) [-M,M]$$
(viewed as a multiset in $\Z$) and let $Q_0$ be the arithmetic progression
$$ Q_0 \coloneqq  [-A^{-2k} M^d, A^{-2k} M^d]$$
(also viewed as a multiset in $\Z$).  Then one has
\begin{equation}\label{emm}
\E_{\vec h \in ([M]^r)^k} \inf_{1 \leq q \leq A^k} d_{TV}(Q(\vec h), Q(\vec h) + q Q_0) \ll_{d,r,k} A^{-1}.
\end{equation}
\end{theorem}

For future reference, we observe that the claim also holds when the arithmetic progressions are viewed as subsets of a cyclic group $\Z/N\Z$ rather than $\Z$, since applying reduction modulo $N$ can only decrease the total variation norm.

\begin{proof}   Let $N \coloneqq  A^{2d} M^d$, then the projection of $\Z$ to $\Z/N\Z$ is injective on $Q(\vec h)$ and $Q(\vec h) + q Q_0$.  Thus, we may interpret these progressions instead as lying in the cyclic group $\Z/N\Z$.  
By Fourier expansion in $\Z/N\Z$, the density function $p_{Q(\vec h)}$ of the multiset $Q(\vec h)$ can be written as
$$ p_{Q(\vec h)}(x) = \frac{1}{N} \sum_{\xi \in \Z/N\Z} e(\xi x / N) \E_{a \in Q(\vec h)} e( - a \xi / N) $$
which factorises as
$$ p_{Q(\vec h)}(x) = \frac{1}{N} \sum_{\xi \in \Z/N\Z} e(\xi x / N) \prod_{j=1}^k D( P_j(\vec h_j) \xi )$$
where $D \colon \Z/N\Z \to [-1,1]$ is the Dirichlet kernel
$$ D(\xi) \coloneqq  \E_{m \in [-M,M]} e( - m \xi / N ).$$
Similarly we have
$$ p_{Q(\vec h)+qQ_0}(x) = \frac{1}{N} \sum_{\xi \in \Z/N\Z} e(\xi x / N) D'(q \xi) \prod_{j=1}^k D(P_j(\vec h_j) \xi)$$
where $D'$ is another Dirichlet kernel, defined by the formula
$$ D'(\xi) \coloneqq  \E_{m \in [-A^{-2k} N, A^{-2k} N]} e(m \xi / N).$$
Subtracting and using the triangle inequality, we conclude that
$$
d_{TV}(Q(\vec h), Q(\vec h)+qQ_0) \leq \sum_{\xi \in \Z/N\Z} |D'(q\xi)-1| \prod_{j=1}^k |D(P_j(\vec h_j) \xi)| $$
and so it suffices to show that
\begin{equation}\label{mos}
\E_{\vec h \in ([M]^r)^k} \inf_{1 \leq q \leq A^k} 
\sum_{\xi \in \Z/N\Z} |D'(q \xi)-1| \prod_{j=1}^k |D(P_j(\vec h_j) \xi)| \ll_{d,r,k} A^{-1}.
\end{equation}
Call $\xi \in \Z/N\Z$ \emph{major arc} if one has
$$ \left\| \frac{q \xi}{N}\right \|_{\R/Z} \leq A^{\sqrt{k}}/N $$
for some $1 \leq q \leq A^{\sqrt{k}}$, and \emph{minor arc} otherwise.  We first consider the contribution to \eqref{mos} of the minor arcs.  Here we bound $|D'(q\xi)-1|$ by $2$, and estimate this contribution by
$$ 2 \sum_\xi^* \prod_{j=1}^k \E_{\vec h_j \in [M]^r} |D(P_j(\vec h_j) \xi)|$$
where $\sum_\xi^*$ denotes a summation over minor arc $\xi$.  By several applications of H\"older's inequality, we may bound this by
$$ 2 \prod_{j=1}^k \left(\sum_\xi (\E_{\vec h_j \in [M]^r} |D(P_j(\vec h_j) \xi)|^2)^{k/2}\right)^{1/k}.$$

We have the following distributional control on the expression $\E_{\vec h_j \in [M]^r} |D(P_j(\vec h_j) \xi)|^2$ appearing above:

\begin{proposition}[Distributional bound]  Let the notation and hypotheses be as above, and let $n$ be a natural number.  Let $1 \leq j \leq k$.  Then there are at most $O_{d,r}( (A2^n)^{O_{d,r}(1)} )$ values of minor arc $\xi \in \Z/N\Z$ for which
$$ \E_{\vec h_j \in [M]^r} |D(P_j(\vec h_j) \xi)|^2 \geq 2^{-n}.$$
Furthermore, there are no such minor arc $\xi$ unless
\begin{equation}\label{nak}
 2^{-n} \leq A^{-k^{1/4}}.
\end{equation}
\end{proposition}

\begin{proof}  For brevity we allow all implied constants in the asymptotic notation to depend on $d$ and $r$. If we have $N \ll (2^n)^{O(1)}$ then \eqref{nak} is trivial (since $N$ is assumed large depending on $d,r,k,A$), and the first claim is also trivial since there are clearly at most $N$ possible choices for $\xi$.  Thus we may assume that we do not have an estimate of the form $N \ll (2^n)^{O(1)}$.

Assume that $\E_{\vec h_j \in [M]^r} |D(P_j(\vec h_j) \xi)|^2 \geq 2^{-n}$ for some natural number $n$.  

We can expand $\E_{\vec h_j \in [M]^r} |D(P_j(\vec h_j) \xi)|^2$ as
$$ \E_{\vec h_j \in [M]^r} |D(P_j(\vec h_j) \xi)|^2 = \E_{\vec h_j \in [M]^r} \E_{m,m' \in [-M,M]} e\left( \frac{m' P_j(\vec h_j) \xi}{N} - \frac{m P_j(\vec h_j) \xi}{N} \right).$$
Applying Lemma \ref{pop}, we conclude that there is a natural number $q_j \ll (A 2^n)^{O(1)}$ such that
$$ \left\| \frac{q_j \xi}{N} \alpha_{j,i_1,\dots,i_r} \right\|_{\R/\Z} \ll \frac{(A2^n)^{O(1)}}{M^{i_1+\dots+i_r+1}}$$
for all the coefficients $\alpha_{j,i_1,\dots,i_r}$ of $P_j$.  In particular, since $P_j$ has at least one non-zero coefficient
$\alpha_{j,i_1,\dots,i_r}$ with $i_1+\dots+i_r=d-1$, we have
\begin{equation}\label{con}
 \left\| \frac{q_j \xi}{N} \alpha_{j,i_1,\dots,i_r} \right\|_{\R/\Z} \ll \frac{(A2^n)^{O(1)}}{N}
\end{equation}
for that coefficient.   This is inconsistent with the minor arc hypothesis unless \eqref{nak} holds.

For fixed $q_j$, the constraint \eqref{con} restricts $\xi$ to at most $O( (A2^n)^{O(1)} )$ possible values; summing over $q_j$, we conclude the proposition.
\end{proof}

Applying the above proposition for each $n$ and summing, we conclude that
\begin{align*}
 \sum^*_{\xi} \left(\E_{\vec h_j \in [M]^r} |D(P_j(\vec h_j) \xi)|^2\right)^{k/2} &\ll_{k} \sum_{n: 2^{-n} \leq A^{-k^{1/4}}} (A2^n)^{O(1)} 2^{-nk/2} \\
&\ll_{k} A^{-1} 
\end{align*}
if $k$ is sufficiently large depending on $d,r$.  Thus the contribution of the minor arcs to \eqref{mos} is acceptable.

It remains to control the contribution to \eqref{mos} of the major arc $\xi$.  Note that the number of major arc $\xi$ is $O( A^{O(\sqrt{k})} )$.  As a consequence, any choice of $\vec h$ and $\xi$ for which 
$$ |D(P_j(\vec h_j) \xi)| \leq A^{-k} $$
for some $1 \leq j \leq k$ gives a negligible contribution to \eqref{mos}.  Thus we may restrict attention to those $\vec h$ and $\xi$ for which $\xi$ is major arc and
$$ |D(P_j(\vec h_j) \xi)| > A^{-k} $$
for all $j=1,\dots,k$.  Summing the Dirichlet series, this implies that
$$ \left\| \frac{P_j(\vec h_j) \xi}{N} \right\|_{\R/\Z} \ll \frac{A^k}{N} $$
for all $j=1,\dots,k$. On the other hand, as $\xi$ is major arc, we have
\begin{equation}\label{qoad}
 \left\| \frac{\xi}{N} - \frac{a_\xi}{q_\xi} \right\|_{\R/Z} \leq A^{\sqrt{k}}/N 
\end{equation}
for some $1 \leq q_\xi \leq A^{\sqrt{k}}$ and $a$ coprime to $q_\xi$.  We also have $P_j(\vec h_j) \ll A^{O(1)} M^{d-1}$.  As $M$ is large, these estimates are only compatible with each other if $q_\xi$ divides $P_j(\vec h_j)$.  In particular, $q_\xi$ divides the greatest common divisor of $P_1(\vec h_1),\dots,P_k(\vec h_k)$.

Suppose that $\vec h$ is such that the greatest common divisor $(P_1(\vec h_1),\dots,P_k(\vec h_k))$ of $P_1(\vec h_1),\dots,P_k(\vec h_k)$ is at most $A^k$.  Denoting this greatest common divisor by $q$, we see from \eqref{qoad} that
$$ |D'(q\xi) - 1| \ll A^{-k + O(\sqrt{k})} $$
for all $\xi$ that have not already been previously eliminated from consideration.  Thus the contribution of these $\vec h$ to \eqref{mos} is acceptable.  Thus we only need to restrict attention to those $\vec h$ for which $(P_1(\vec h_1),\dots,P_k(\vec h_k)) > A^k$.  For this case, we bound $|D'(q \xi)-1|$ by $2$ and $D(P_j(\vec h_j) \xi)$ by one, and use the fact that there are only $O(A^{O(\sqrt{k})})$ major arc $\xi$, and bound this contribution to \eqref{mos} by
\begin{equation}\label{mos-left}
 \ll A^{O(\sqrt{k})} \E_{\vec h \in ([M]^r)^k} 1_{(P_1(\vec h_1),\dots,P_k(\vec h_k)) > A^k}.
\end{equation}
Observe from the Chinese remainder theorem and the divisor bound that for any $q > A^k$, the constraint $q|P_j(\vec h_j)$ restricts $\vec h_j$ to $O_{\eps}( A^{O(1)} q^{-1+\eps} )$ possible values for any $\eps>0$.  Setting $\eps = 1/2$ (say), this implies that there are $O( A^{O(1)} q^{-1/2} M^r)$ choices of $\vec h_j$ for which $q | P_j(\vec h_j)$.  Thus, the greatest common divisor $(P_1(\vec h_1),\dots,P_k(\vec h_k))$ can equal $q$ for at most $O( A^{O(1)} q^{-1/2} M^r)^k$ choices of $\vec h$.  Thus \eqref{mos-left} is bounded by
$$
\ll A^{O(\sqrt{k})} \sum_{q > A^k} \left(O( A^{O(1)} q^{-1/2})\right)^k $$
which is acceptable since $k$ is sufficiently large depending on $d,r$.
\end{proof}

Next, we recall a key consequence of the concentation theory developed in \cite{tz-concat}.

\begin{theorem}[Qualitative Bessel inequality for box norms]\label{besu-box}  Let $d$ be a positive integer.  For each $1 \leq j \leq d$, let $(Q_{i,j})_{i \in I}$ be a finite family of progressions 
$$Q_{i,j} = a_{i,j,1} [-M_{i,j,1},M_{i,j,1}] + \dots + a_{i,j,r_{i,j}} [-M_{i,j,r_{i,j}}, M_{i,j,r_{i,j}}]$$
with ranks $r_{i,j}$ at most $r$, in a cyclic group $\Z/N\Z$. Let $f$ lie in the unit ball of $L^\infty(\Z/\N\Z)$, and suppose that
$$
\E_{i,j \in I} \|f\|_{\Box^{d^2}_{(\eps Q_{i,k}+\eps Q_{j,l})_{1 \leq k \leq d, 1 \leq l \leq d}}(\Z/N\Z)} \leq \eps $$
for some $\eps > 0$.  Then
$$
\E_{i \in I} \|f\|_{\Box^{d}_{Q_{i,1},\dots,Q_{i,d}}(\Z/N\Z)} \leq c(\eps) 
$$
where $c\colon (0,+\infty) \to (0,+\infty)$ is a function such that $c(\eps) \to 0$ as $\eps \to 0$.  Furthermore, $c$ depends only on $r$ and $d$.
\end{theorem}

\begin{proof} See \cite[Theorem 1.28]{tz-concat}.
\end{proof}

We combine this theorem with Theorem \ref{mung} to obtain our first result controlling an averaged Gowers norm by global Gowers norms.

\begin{theorem}[Global Gowers norms control averaged Gowers norms]\label{avg-norm}  Let $d,r,D \geq 1$ be natural numbers, let $d_0$ be an integer with $1 \leq d_0 \leq d$, and suppose that $D_* \geq 1$ is a natural number that is sufficiently large depending on $d,r,D$.  Let $\eps > 0$, and let $\delta>0$ be a quantity that is sufficiently small depending on $d,r,D,\eps$.  Let $A$ be sufficiently large depending on $d,r,D,D_*,\eps,\delta$.  For each $j=1,\dots,D$, let $P_j \in \Z[\h_1,\dots,\h_r]$ be a polynomial of degree between $d_0-1$ and $d-1$ inclusive with coefficients that are integers of magnitude at most $A$.  Let $M$ be sufficiently large depending on $d,r,D,D_*,\eps,\delta,A$.  Let $N$ be a quantity larger than $A^{-1} M^d$.  Let $f \colon \Z/N\Z \to \R$ be a function bounded in magnitude by $1$ such that
\begin{equation}\label{cat}
 \|f\|_{U^{D_*}_{q[A^{-2D_*} M^{d_0}]}} \leq \delta
\end{equation}
for all $1 \leq q \leq A^{D_*}$.  Then
\begin{equation}\label{cat-con}
 \E_{\vec h \in [M]^r} \|f\|_{\Box^D_{P_1(\vec h)[-M,M],\dots,P_D(\vec h)[-M,M]}} \leq \eps.
\end{equation}
Here the arithmetic progressions are viewed as multisets in $\Z/N\Z$.
\end{theorem}

A key point here is that $\delta$ does not depend on $A$.

\begin{proof}  Let $k$ be a power of two that is sufficiently large depending on $d,r,D$; we assume $D_*$ sufficiently large depending on $k$.  Let $\sigma > 0$ be a quantity that is sufficiently small depending on $d,r,k,D,\eps$; we assume $\delta$ sufficiently small depending on $d,r,k,D,\sigma$.  We will show that
\begin{equation}\label{frog}
 \E_{\vec h_1,\dots,\vec h_k \in [M]^r} \|f\|_{\Box^{D^k}_{(\sum_{j=1}^k P_{i_j}(\vec h_j)[-\sigma M,\sigma M])_{i_1,\dots,i_k \in [D]}}} \ll \sigma;
\end{equation}
the claim will then follow from $\log_2 k$ applications of Theorem \ref{besu-box}.

We abbreviate $\vec i \coloneqq  (i_1,\dots,i_k)$, $\vec h \coloneqq  (\vec h_1,\dots, \vec h_k)$, and 
\begin{equation}\label{qih}
Q_{\vec i}(\vec h) \coloneqq  \sum_{j=1}^k P_{i_j}(\vec h_j)[-\sigma M,\sigma M].
\end{equation}
The left-hand side of \eqref{frog} can then be written as
$$ \E_{\vec h \in [M_0]^{kr}} \|f\|_{\Box^{D^k}_{(Q_{\vec i}(\vec h))_{\vec i \in [D]^k}}}.$$
By H\"older's inequality, we may upper bound this by
$$ \left(\E_{\vec h \in [M_0]^{kr}} \|f\|_{\Box^{D^k}_{(Q_{\vec i}(\vec h))_{\vec i \in [D]^k}}}^{2^{D^k}}\right)^{1/2^{D^k}}.$$
It will thus suffice to show that
$$ \|f\|_{\Box^{D^k}_{(Q_{\vec i}(\vec h))_{\vec i \in [D]^k}}}^{2^{D^k}} \ll \sigma^{2^{D^k}}$$
for all but $O( \sigma^{2^{D^k}} M_0^{kr} )$ of the $\vec h$ in $[M_0]^{kr}$.

For each $\vec i \in [D]^k$, the multiset $Q_{\vec i}(\vec h)$ is constructed using the polynomials $P_{i_1},\dots,P_{i_k}$.  By the pigeonhole principle, there exists $d_0 \leq d_1 \leq d$ (depending on $\vec i$) such that at least $k/d$ of the polynomials $P_{i_1},\dots,P_{i_k}$ have degree exactly $d_1-1$.  Let $P_{i_{j_1}},\dots,P_{i_{j_{k'}}}$ denote these polynomials, then we can write
$$ Q_{\vec i}(\vec h) = Q'_{\vec i}(\vec h) + Q''_{\vec i}(\vec h)$$
where
$$ Q'_{\vec i}(\vec h) \coloneqq  \sum_{l=1}^{k'} P_{i_{j_l}}(\vec h_{j_l})[-\sigma M,\sigma M]$$
and
$$ Q''_{\vec i}(\vec h) \coloneqq  \sum_{j \in \{1,\dots,k\} \backslash \{j_1,\dots,j_{k'}\}} P_{i_{j}}(\vec h_{j})[-\sigma M,\sigma M].$$
Applying Theorem \ref{mung} and Markov's inequality with these polynomials $P_{i_1},\dots,P_{i_k}$ (and $d$ replaced by $d_1$), we conclude that
for all but  $O( \sigma^{2^{D^k}} M_0^{kr} )$ of the $\vec h$ in $[M_0]^{kr}$, there exists $1 \leq q_{\vec i} \leq A^k$ (depending on $\vec h$) for each $\vec i \in [D]^k$ such that
$$
d_{TV}(Q'_{\vec i}(\vec h), Q'_{\vec i}(\vec h) + q_{\vec i} [-A^{-2k} N^{d_1/d}, A^{-2k} N^{d_1/d}]) \ll_{d,r,k,D,\sigma} A^{-1}.
$$
Applying \eqref{dtv}, we conclude that
$$
d_{TV}(Q_{\vec i}(\vec h), Q_{\vec i}(\vec h) + q_{\vec i} [-A^{-2k} N^{d_1/d}, A^{-2k} N^{d_1/d}]) \ll_{d,r,k,D,\sigma} A^{-1}.
$$
If $q$ is the product of all the $q_{\vec i}$, then from direct computation
$$ d_{TV}( q_{\vec i} [-A^{-2k} N^{d_1/d}, A^{-2k} N^{d_1/d}], q_{\vec i} [-A^{-2k} N^{d_1/d}, A^{-2k} N^{d_1/d}] + q [A^{-2D_*} N^{d_0/d}] ) \leq A^{-1}$$
and hence by \eqref{dtv-2}
$$ d_{TV}(Q_{\vec i}(\vec h), Q_{\vec i}(\vec h) + q [A^{-2D_*} M^{d_0}]) \ll_{d,r,k,D,\sigma} A^{-1}.$$
From many applications of \eqref{say}, we thus have
$$ 
\|f\|_{\Box^{D^k}_{(Q_{\vec i}(\vec h))_{\vec i \in [D]^k}}}^{2^{D^k}} = \|f\|_{\Box^{D^k}_{(Q_{\vec i}(\vec h)+ q [A^{-2D_*} M^{d_0}])_{\vec i \in [D]^k}}}^{2^{D^k}} + O_{d,r,k,D,\sigma}(A^{-1}) 
$$
so it will suffice to show that
$$  \|f\|_{\Box^{D^k}_{(Q_{\vec i}(\vec h)+ q [A^{-2D_*} M^{d_0}])_{\vec i \in [D]^k}}}^{2^{D^k}} \ll  \sigma^{2^{D^k}}.$$
The left-hand side expands as
$$ \E_{a_{0,\vec i}, a_{1,\vec i} \in Q_{\vec i}(\vec h) \forall \vec i \in [D]^k}
\left\langle T^{\sum_{\vec i \in [D]^k} a_{\omega_{\vec i},\vec i}} f \right\rangle_{U^{D^k}(q [A^{-D_*} M^{d_0}])} $$
where $T^h f(x) \coloneqq  f(x+h)$ is the shift of $f$ by $h$.  By the Cauchy-Schwarz-Gowers inequality \eqref{csg} and the translation-invariance of the Gowers norms, we can bound this by
$$ \| f \|_{U^{D^k}(q [A^{-2D_*} M^{d_0}])}$$
and the desired bound now follows from \eqref{cat} and the monotonicity of the Gowers norms.
\end{proof}

Theorem \ref{avg-norm} is not directly applicable to our applications involving primes, because of the requirement that the function $f$ is bounded in magnitude by one.  In principle, the ``transference principle'' introduced in \cite{gt-primes} should be able to relax this requirement to allow for unbounded $f$ (so long as $f$ is still bounded pointwise by a suitably ``pseudorandom'' majorant), but this turns out to require a fair amount of additional argument.  To begin this task, we present a ``dual'' form of Theorem \ref{avg-norm}, which roughly speaking asserts that dual functions associated to averaged Gowers norms can be approximated by polynomial combinations of dual functions associated with global Gowers norms.

\begin{theorem}[Dual function approximation]\label{dfa}  Let $d,r,D \geq 1$ be natural numbers; let $d_0$ be an integer with $1 \leq d_0 \leq d$.  Suppose that $D_* \geq 1$ is a natural number that is sufficiently large depending on $d,r,D$.  Let $\eps > 0$, and let $K>0$ be a quantity that is sufficiently large depending on $d,r,D,D_*,\eps$.  Let $A$ be sufficiently large depending on $d,r,D,D_*,\eps,K$.  For each $j=1,\dots,D$, let $P_j \in \Z[\h_1,\dots,\h_r]$ be a polynomial of degree between $d_0-1$ and $d-1$ inclusive
 with coefficients that are integers of magnitude at most $A$.  Let $M$ be sufficiently large depending on $d,r,D,D_*,\eps,K,A$, and let $M_0, N$ be such that $A^{-1} M \leq M_0 \leq AM$ and $N \geq A^{-1} M^d$.  Let $f \colon \Z/N\Z \to \R$ be a function bounded in magnitude by $1$, and define the ``averaged dual function''
\begin{equation}\label{stim}
 F \coloneqq 
\E_{\vec h \in [M_0]^r} {\mathcal D}^D_{P_1(\vec h)[-M,M],\dots,P_D(\vec h)[-M,M]}(f).
\end{equation}
Then one can find a function $\tilde F \colon \Z/N\Z \to [-2,2]$ with
\begin{equation}\label{Trick}
 \|F - \tilde F \|_{L^2(\Z/N\Z)} \leq \eps
\end{equation}
such that $\tilde F$ is a linear combination
\begin{equation}\label{fkcf}
 \tilde F = \sum_{i=1}^k c_i F_i
\end{equation}
with $k \leq K$, and for each $i=1,\dots,k$, $c_i$ is a scalar with $|c_i| \leq 1$, and $F_i \colon \Z/N\Z \to \C$ is a function of the form
$$ F_i = \prod_{j=1}^{k_i} {\mathcal D}^{D_*}_{q_{i,j}[A^{-2D_*} M^{d_0}]}( f_{i,j} )$$
with $k_i \leq K$, and for each $j=1,\dots,k$, $q_{i,j}$ is a natural number with $q_{i,j} \leq A^{D_*}$, and $f_{i,j} \colon \Z/N\Z \to [-1,1]$ is a function bounded in magnitude by $1$.  As before, the progressions are viewed as multisets in $\Z/N\Z$.
\end{theorem}

\begin{proof}  Let $d,r,d_0,D,D_*,\eps,K,A,P_j,N,M_0,M,f,F$ be as above.
For any given $K'$, let ${\mathcal F}_{K'}$ denote the collection of all the functions $\tilde F$ that have a decomposition \eqref{fkcf} with the indicated properties, but with $K$ replaced by $K'$ throughout.  Our task is then to show that there exists $\tilde F \in {\mathcal F}_K$ such that $\|F - \tilde F \|_{L^2(\Z/N\Z)} \leq \eps$.

Let $\delta > 0$ be a sufficiently small quantity depending on $d,r,D,\eps$, and let $\sigma > 0$ be sufficiently small depending on $d,r,D,\eps,\delta$.  We will prove the following \emph{energy decrement} claim: if $K' \geq 0$ is a natural number, and $\tilde F \in {\mathcal F}_{K'}$ is such that either
\begin{equation}\label{footf} 
|\langle \tilde F, F-\tilde F \rangle_{L^2(\Z/N\Z)}| > \delta
\end{equation}
or
\begin{equation}\label{ftf}
 \|F - \tilde F\|_{U^{D_*}_{q[A^{-2D_*} M^{d_0}]}} > \delta
\end{equation}
for some $q \leq A^{D_*}$, then whenever $K''$ is sufficiently large depending on $\delta,\sigma,K'$, and $A$ sufficiently large depending on $D_*,\delta,\sigma,K',K''$, there exists $\tilde F' \in {\mathcal F}_{K''}$ such that
$$
\|F - \tilde F' \|_{L^2(\Z/N\Z)}^2 \leq \|F - \tilde F \|_{L^2(\Z/N\Z)}^2  - \sigma.$$
Applying this energy decrement claim iteratively starting from $K'=0$ and $\tilde F=0$ (which implies in particular that $\|F - \tilde F \|_{L^2(\Z/N\Z)}^2 \leq 1$), we obtain after at most $1/\sigma$ iterations that for $K$ large enough, there exists $\tilde F \in {\mathcal F}_K$ such that
\begin{equation}\label{ort}
|\langle \tilde F, F-\tilde F \rangle_{L^2(\Z/N\Z)}| \leq \delta
\end{equation}
and
$$ \|F - \tilde F\|_{U^{D_*}_{q[A^{-2D_*} M^{d_0}]}} \leq \delta
$$
for all $q \leq A^{D_*}$.  Applying Theorem \ref{avg-norm} (with $f$ set equal to $\frac{1}{3} (F - \tilde F)$, and adjusting $\eps$ and $\delta$ slightly), we conclude that
\begin{equation}\label{mod}
 \E_{\vec h \in [M_0]^r} \|F-\tilde F\|_{\Box^D_{P_1(\vec h)[-M,M],\dots,P_D(\vec h)[-M,M]}} \leq \eps^2/2.
\end{equation}
We can then use \eqref{stim}, the Cauchy-Schwarz-Gowers inequality \eqref{csg}, and \eqref{mod} to bound
\begin{align*}
\langle F, F - \tilde F \rangle_{L^2(\Z/N\Z)}
&= 
\E_{\vec h \in [M_0]^r} \langle {\mathcal D}^D_{P_1(\vec h)[-M,M],\dots,P_D(\vec h)[-M,M]}(f), F - \tilde F \rangle_{L^2(\Z/N\Z)} \\
&\leq
\E_{\vec h \in [M_0]^r} \| F - \tilde F \|_{\Box^D_{P_1(\vec h)[-M,M],\dots,P_D(\vec h)[-M,M]}}\\
&\leq \eps^2/2
\end{align*}
and hence by \eqref{ort} we conclude \eqref{Trick}.

It remains to prove the energy decrement claim. Let $K' \geq 0$, and let $\tilde F \in {\mathcal F}_{K'}$ obey either \eqref{footf} or \eqref{ftf}.  First suppose that \eqref{footf} holds.  From the Cauchy-Schwarz inequality this implies that 
\begin{equation}\label{caf}
\|\tilde F\|_{L^2(\Z/N\Z)} \geq \delta/3
\end{equation}
 (since we can bound the $L^\infty$, and hence the $L^2$, norm of $F - \tilde F$ by $3$).   If we let $c \tilde F$ be the orthogonal projection of $F$ to the one-dimensional space spanned by $\tilde F$, thus
$$ c \coloneqq  \frac{\langle F, \tilde F \rangle_{L^2(\Z/N\Z)}}{\|\tilde F\|_{L^2(\Z/N\Z)}^2}.$$
From \eqref{caf} we see in particular that
\begin{equation}\label{cbig}
|c| \leq \frac{100}{\delta}.
\end{equation}
Also, since
$$ \langle \tilde F, F - \tilde F \rangle_{L^2(\Z/N\Z)} = \langle \tilde F, c \tilde F - \tilde F \rangle_{L^2(\Z/N\Z)}$$
we see that
$$ |c-1| \geq \frac{\delta}{10}.$$
By Pythagoras' theorem, we conclude that
$$ \|F - c\tilde F \|_{L^2(\Z/N\Z)}^2 \leq \|F - \tilde F \|_{L^2(\Z/N\Z)}^2 - \frac{\delta^3}{100}.$$
We would like to take $c \tilde F$ to be the function $\tilde F'$, but there is the technical difficulty that $c \tilde F$ need not take values in $[-2,2]$.  However, from \eqref{cbig} we see that $c \tilde F$ takes values in $[-200/\delta,200/\delta]$.  We now use the Weierstrass approximation theorem to find a polynomial $P$ of degree $O_{\delta,\sigma}(1)$ and coefficients $O_{\delta,\sigma}(1)$ such that
$$ |P(t) - \max(\min(t,-1),1)| \leq \sigma^2 $$
for all $t \in [-200/\delta, 200/\delta]$.  Then $P( c \tilde F )$ takes values in $[-2,2]$ and lies in ${\mathcal F}_{K''}$ for $K''$ large enough depending on $\delta,\sigma,K'$, and
$$ \|P(c\tilde F) - \max(\min(c\tilde F,-1),1) \|_{L^\infty(\Z/N\Z)} \leq \sigma^2 $$
and thus by the triangle inequality
\begin{align*}
 \|F - P(c\tilde F) \|_{L^2(\Z/N\Z)}^2 
&\leq \|F - \max(\min(\tilde F,-1),1) \|_{L^2(\Z/N\Z)}^2 - \sigma \\
&\leq \|F - \tilde F \|_{L^2(\Z/N\Z)}^2 - \sigma 
\end{align*}
where the last line comes from the fact that $F$ takes values in $[-1,1]$ and the map $t \mapsto \max(\min(t,-1),1)$ is a contraction. 
This gives the required claim in the case that \eqref{footf} holds.

Now suppose that \eqref{ftf} holds for some $q \leq A^{D_*}$.
If we write $g \coloneqq  \frac{1}{3} (F-\tilde F)$, then $g$ takes values in $[-1,1]$, and
$$ \|g\|_{U^{D_*}_{q[A^{-2D_*} M^{d_0}]}} > \delta/3.$$
This implies that
$$ \langle g, {\mathcal D}^{D_*}_{q[A^{-2D_*} M^{d_0}]} \rangle_{L^2(\Z/N\Z)} > (\delta/3)^{2^{D_*}}.$$
and thus
$$ \langle F - \tilde F, {\mathcal D}^{D_*}_{q[A^{-2D_*} M^{d_0}]} g \rangle_{L^2(\Z/N\Z)} > (\delta/3)^{2^{D_*}}/3.$$
If we then set 
$$ \tilde F_1 \coloneqq  \tilde F - \sigma^{2/3} {\mathcal D}^{D_*}_{q[A^{-2D_*} M^{d_0}]} g$$
(say) then from the cosine rule we have
$$ \|F - \tilde F_1 \|_{L^2(\Z/N\Z)}^2 \leq \|F - \tilde F \|_{L^2(\Z/N\Z)}^2 - 2\sigma$$
(say).  We would like to take $\tilde F_1$, but again there is the slight problem that $\tilde F_1$ can take values a little bit outside of $[-2,2]$.  However if one sets $\tilde F \coloneqq  P(\tilde F_1)$ where $P$ is the polynomial constructed previously, one obtains the desired claim.
\end{proof}

\section{Averaged Gowers uniformity of a $W$-tricked von Mangoldt function}\label{agu}

In this section we require the following parameters, chosen in the following order.
\begin{itemize}
\item First, we let $d,r,D \geq 1$ be arbitrary natural numbers.  Let $d_0$ be an integer with $1 \leq d_0 \leq d$.
\item Then, we let $D_*$ be a natural number that is sufficiently large depending on $d,r,D$.
\item Then, we choose $\kappa > 0$ to be a real number that is sufficiently small depending on $d,r,D,D_*$.
\item We let $N'$ be a large integer (going to infinity), and let $w = w(N')$ be a sufficiently slowly growing function of $N'$ (the choice of $w$ can depend on $d,r,D,D_*,\kappa$).  We use $o(1)$ to denote any quantity that goes to zero as $N' \to \infty$.  
\item We set
\begin{equation}\label{wap}
W \coloneqq  \prod_{p \leq w} p.
\end{equation}
and $N \coloneqq  \lfloor N'/W\rfloor$.  We also let $b \in [W]$ be a natural number coprime to $W$.
\item We set
$$ A \coloneqq  W^{1/\kappa}$$
and
\begin{equation}\label{rk}
 R \coloneqq  N^\kappa.
\end{equation}
\item We select a quantity $M$ such that
$$ \log^{1/\kappa} N \leq M \leq (A N)^{1/d} $$
\item Finally, for each $1 \leq j \leq J$, we select a polynomial $P_j \in \Z[\h_1,\dots,\h_r]$ of degree between $d_0-1$ and $d-1$ inclusive with coefficients that are integers of magnitude at most $A$.  
\end{itemize}

The reader may wish to keep in mind the hierarchy
$$ d,r,D \ll D_* \ll 1/\kappa \ll W \ll A \ll R \ll N$$
and
$$ A, \log^{1/\kappa} N \ll M \ll N$$
where the quantities $W,A,R,M,N$ go to infinity (at different rates) as $N \to \infty$, while $d,r,D,D_*,\kappa$ do not depend on $N$.  Note that we allow

For each $b \in [W]$ coprime to $W$, define the normalised von Mangoldt functions $\Lambda'_{b,W} \colon \Z/N\Z \to \R$ by the formula
\begin{equation}\label{bawn}
 \Lambda'_{b,W}(x) \coloneqq  \frac{\phi(W)}{W} \Lambda'(Wx+b)
\end{equation}
for $x \in [R,N]$ (where we embed $[R,N]$ into $\Z/N\Z$), with $\Lambda'_{b,W}(x)$ equal to zero for other choices of $x$.  Here $\Lambda'(n)$ is the restriction of the von Mangoldt function to primes, thus $\Lambda'(p) = \log p$ for primes $p$, and $\Lambda'(n)=0$ for non-prime $n$.

The objective of this section is to establish the following bound:

\begin{theorem}[Averaged Gowers uniformity of von Mangoldt]\label{l1}  Let the notation and hypotheses be as above.
If $f \colon \Z/N\Z \to \R$ is any function with the pointwise bound $0 \leq f \leq \Lambda'_{b,W}$, then there exists a function $g \colon \Z/N\Z \to \R$ with the pointwise bounds $0 \leq f' \ll 1$ such that
\begin{equation}\label{hmr}
 \E_{\vec h \in [M]^r} \| f - f' \|_{\Box^D_{P_1(\vec h)[-M,M], \dots, P_D(\vec h)[-M,M]}} = o(1).
\end{equation}
Furthermore, in the special case where $d_0=d$, $M \geq (A^{-1} N)^{1/d}$, and $f = \Lambda'_{b,W}$, we may take $f'=1$.
\end{theorem}

The approximating function $f'$ is known as a \emph{dense model} for $f$ in the literature.  The above theorem can be compared with the bound $\| f-f'\|_{U^D(\Z/N\Z)} = o(1)$ established in \cite{gt-primes}, and the results in \cite{gt-linear} which imply that $f'$ can be taken to equal $1$ when $f = \Lambda'_{b,W}$.  Also, the arguments in \cite{tz-pattern} give a version of the first part of this theorem in the case that $M$ is a very small power of $N$ (in particular, much smaller than $R$), and the subsequent arguments in \cite{tz-narrow} extend this to cover the regime where $M$ grows slower than any power of $N$ but faster than any power of $\log N$.  These restrictions on $M$ arose from a certain ``clearing denominators'' step encountered when dealing with products of many dual functions associated to averaged local Gowers norms; they are circumvented in this paper by application of the concatenation machinery to replace these norms with more traditional Gowers norms that do not require the ``clearing denominators'' method in order to handle products of dual functions.

\begin{remark}\label{ms}  In \cite[Theorem 9]{tz-narrow}, in the notation of the current paper, the slightly weaker bound
$$
 \E_{\vec h \in [M]^r} \| f - f' \|_{\Box^D_{P_1(\vec h)[-M,M], \dots, P_D(\vec h)[-M,M]}} \leq \eps + o(1)
$$
was established for any fixed $\eps > 0$ assuming that $M = \log^L N$ for some $L = L(\eps)$ that was sufficiently large depending on $\eps$.  Theorem \ref{l1} allows one to now take $L = 1/\kappa$ independent of $\eps$.  By \cite[Remark 4]{tz-narrow}, this allows one to also take $L$ independent of $\eps$ in \cite[Theorem 5]{tz-narrow}, which gives Theorem \ref{main-sub}.
\end{remark}

Morally speaking, Theorem \ref{l1} ought to follow from the results in \cite{gt-primes}, \cite{gt-linear} mentioned above after applying Theorem \ref{avg-norm}, but we run into the familiar difficulty that the function $\Lambda'_{b,W}-1$ is not bounded.  In the final part of Theorem \ref{l1}, the conditions on $d_0$ and $M$ should be dropped, but this requires control on correlations of $\Lambda$ with nilsequences on short intervals, and such control is not currently available in the literature.

We now begin the proof of Theorem \ref{l1}. As in \cite{gt-primes}, \cite{tz-pattern}, we envelop $\Lambda'_{b,W}$ (and hence $f$) by a pseudorandom majorant $\nu = \nu_b:\Z/N\Z \to \R^+$, defined as follows.  Let $\chi \colon \R \to \R$ be a fixed smooth even function that vanishes outside of $[-1,1]$, positive at $0$, and obeys the normalisation
$$ \int_0^1 |\chi'(t)|^2\ dt = 1.$$
We allow all implied constants to depend on $\chi$.  We then set
\begin{equation}\label{nu-def}
 \nu(x) = \nu_b(x) \coloneqq  \frac{\phi(W)}{W} \log R \left( \sum_{m|Wx+b} \mu(m) \chi\left(\frac{\log m}{\log R}\right) \right)^2 
\end{equation}
for all $x \in [N]$.  Comparing this with \eqref{bawn} and \eqref{rk} we conclude the pointwise bound
\begin{equation}\label{kap}
 0 \leq f(x) \leq \Lambda'_{b,W}(x) \ll_\kappa \nu(x) 
\end{equation}
for all $x \in \Z/N\Z$, since $\Lambda'_{b,W}(x)$ is only non-vanishing when $x \in [R,N]$ and $Wx+b$ is prime, in which case the only non-zero summand in \eqref{nu-def} comes from the $m=1$ term.

It was observed in \cite{gt-primes} that
\begin{equation}\label{noo}
\E_{x \in \Z/N\Z} \nu(x) = 1+o(1).
\end{equation}
In fact we have many further ``pseudorandomness'' properties of $\nu$; roughly speaking, any ``non-degenerate'' multilinear correlation of $\nu$ with itself (with reasonable bounds on coefficients) will be $1+o(1)$, as long as the complexity of the correlation is small compared to $1/\kappa$.  Here is one specific instance of the principle we will need:

\begin{proposition}[Gowers uniformity of $\nu-1$]\label{gu}  Let $K$ be a natural number independent of $N$ (but which can depend on $d,r,D,D_*,\kappa$).  Let $q \leq A^{KD_*}$ be a natural number. Then we have
\begin{equation}\label{gu-1}
\E_{\vec h, \vec h' \in [M]^r} \|\nu-1\|_{\Box^{D_*+2D}_{q[A^{-3KD_*}M^{d_0}], \dots, q[A^{-3KD_*}M^{d_0}], P_1(\vec h)[-M,M],\dots,P_D(\vec h)[-M,M],P_1(\vec h')[-M,M],\dots,P_D(\vec h')[-M,M]}}^{2^{D_*+2D}} = o(1).
\end{equation}
In particular, by monotonicity \eqref{mono}, permutation symmetry, and H\"older's inequality one has the estimates
$$
\|\nu-1\|_{\Box^{D_*}_{q[A^{-3KD_*}M^{d_0}], \dots, q[A^{-3KD_*}M^{d_0}]}} = o(1),$$
$$
\E_{\vec h \in [M]^r} \|\nu-1\|_{\Box^{D_*+D}_{q[A^{-3KD_*}M^{d_0}], \dots, q[A^{-3KD_*}M^{d_0}], P_1(\vec h)[-M,M],\dots,P_D(\vec h)[-M,M]}}^{2^{D_*+D}} = o(1),$$
and
$$
\E_{\vec h, \vec h' \in [M]^r} \|\nu-1\|_{\Box^{2D}_{P_1(\vec h)[-M,M],\dots,P_D(\vec h)[-M,M],P_1(\vec h')[-M,M],\dots,P_D(\vec h')[-M,M]}}^{2^{2D}} = o(1).$$
\end{proposition}

\begin{proof}  See Appendix \ref{polyapp}.
\end{proof}

Next, we give a variant of the estimate in \cite[Proposition 6.2]{gt-primes}.

\begin{proposition}[$\nu-1$ orthogonal to dual functions]\label{vari}  Let the notation and hypotheses be as in Theorem \ref{l1}.  Let $K,K'$ be natural numbers independent of $N$ (but which can depend on $d,r,D,D_*,\kappa$).   
For each $j=1,\dots,K$, let $\vec f_j = (f_{j,\omega})_{\omega \in \{0,1\}^{D_*} \backslash \{0\}^{D_*}}$ be a tuple of functions $f_{j,\omega} \colon \Z/N\Z \to [-1,1]$, and $q_j \leq A^{D_*}$ be a natural number, and write
$$ F_j \coloneqq  {\mathcal D}^{D_*}_{q_{j}[A^{-2D_*}M^{d_0}]}( f_{j} ).$$
For each $j'=1,\dots,K'$, let $\vec g_{j'} = (g_{j',\omega})_{\omega \in \{0,1\}^D \backslash \{0\}^D}$ be a tuple of functions $g_{j',\omega} \colon \Z/N\Z \to [-1,1]$, and write
$$ G_{j'} \coloneqq  \E_{\vec h \in [M]^r} {\mathcal D}^D_{P_1(\vec h)[-M,M],\dots,P_D(\vec h)[-M,M]}(\vec g_{j'}).$$
Then
\begin{equation}\label{xfa}
 \E_{x \in \Z/N\Z} (\nu(x)-1) \left(\prod_{j=1}^{K} F_j(x)\right) \left(\prod_{j'=1}^{K'} G_{j'}(x)\right) = o(1).
\end{equation}
\end{proposition}

A key technical point here (as in \cite{gt-primes}) is that the parameters $K,K'$ are allowed to be large compared to $1/\kappa$ (though $K, K'$ will still be small compared to $W,A,R,M,N$).  The potentially large nature of $K,K'$ requires one to proceed carefully, as the pseudorandomness properties of $\nu$ do not allow us to directly control averages involving a number of forms that are comparable or larger to $1/\kappa$.  On the other hand, Proposition \ref{vari} is easier to prove in one respect than \cite[Proposition 6.2]{gt-primes}, because the functions $f_{j,\omega}, g_{j',\omega}$ are assumed to be bounded in magnitude by $1$ rather than $\nu+1$.  We are able to use this stronger hypothesis due to our use of the Conlon-Fox-Zhao densification machinery (which was not available at the time that \cite{gt-primes} was written) from \cite{cfz} later in this paper.

\begin{proof}  We will first prove \eqref{xfa} for small values of $K'$, specifically $K'=0,1,2$, and then use Theorem \ref{dfa} to conclude the case of general $K'$. 

We begin with the $K'=0$ case.  In this case, we may adapt the arguments used to prove \cite[Proposition 6.2]{gt-primes}.  First, we write the left-hand side of \eqref{xfa} as
$$
\E_{x \in \Z/N\Z} (\nu(x)-1) \prod_{j=1}^K \prod_{\omega \in \{0,1\}^{D_*}\backslash \{0\}^{D_*}}
\E_{\vec h_j \in Q_j^{D_*}} f_{j,\omega}( x + \omega \cdot \vec h_j ) $$
where $Q_j$ is the multiset
$$ Q_j \coloneqq  q_j [A^{-2D_*}M^{d_0}] - q_j [A^{-2D_*}M^{d_0}]$$
and $\cdot$ denotes the usual dot product:
$$ (\omega_1,\dots,\omega_{D_*}) \cdot (h_{j,1},\dots,h_{j,D_*}) \coloneqq  \sum_{i=1}^{D_*} \omega_i h_{j,i}.$$
We now ``clear denominators'' by writing $q \coloneqq  \prod_{j=1}^K q_j$, and introducing the multiset
$$ Q \coloneqq  q [A^{-3KD_*}M^{d_0}] - q [A^{-3KD_*}M^{d_0}].$$
Note that $q \leq A^{KD_*}$, which implies that
$$ d_{TV}(Q_j, Q_j+h) = o(1)$$
for all $h \in Q$, thus
$$
\E_{\vec h_j \in Q_j^{D_*}} f_{j,\omega}( x + \omega \cdot \vec h_j )
= \E_{\vec H_j \in Q_j^{D_*}} f_{j,\omega}( x + \omega \cdot \vec H_j + \omega \cdot \vec h) + o(1)$$
for all $j \in Q$ and $\vec h \in Q^{D_*}$.
From \eqref{noo} we have $\E_{x \in \Z/N\Z} \nu(x)+1 = 2+o(1)$, so we can write the left-hand side of \eqref{xfa} as
$$
\E_{x \in \Z/N\Z} (\nu(x)-1) \prod_{j=1}^K \prod_{\omega \in \{0,1\}^{D_*}\backslash \{0\}^{D_*}}
\E_{\vec H_j \in Q_j^{D_*}} f_{j,\omega}( x + \omega \cdot \vec H_j + \omega \cdot \vec h) + o(1);$$
averaging over $\vec h$, we can also write this left-hand side as
$$
\E_{x \in \Z/N\Z} \E_{\vec h \in Q^{D_*}} (\nu(x)-1) \prod_{j=1}^K \prod_{\omega \in \{0,1\}^{D_*}\backslash \{0\}^{D_*}}
\E_{\vec H_j \in Q_j^{D_*}} f_{j,\omega}( x + \omega \cdot \vec H_j + \omega \cdot \vec h) + o(1).$$
We can rewrite this as
$$
\E_{\vec H \in \prod_{j=1}^K Q_j^{D_*}}  \langle (f_{\omega; \vec H})_{\omega \in \{0,1\}^{D_*}} \rangle_{U^{D_*}_{q[A^{-3KD_*}M^{d_0}]}} + o(1)$$
where
\begin{equation}\label{f0h}
 f_{\{0\}^{D_*}; \vec H}(x) \coloneqq  \nu(x)-1
\end{equation}
and
\begin{equation}\label{foh}
 f_{\omega;\vec H}(x) \coloneqq  \prod_{j=1}^K f_{j,\omega}( x + \omega \cdot \vec H_j)
\end{equation}
for $\omega \in \{0,1\}^{D_*} \backslash \{0\}^{D_*}$ and $\vec H = (\vec H_1,\dots,\vec H_k)$.  By the Cauchy-Schwarz-Gowers inequality \eqref{csg}, we can bound this by
$$
\E_{\vec H \in \prod_{j=1}^K Q_j^{D_*}}  \prod_{\omega \in \{0,1\}^{D_*}} \| f_{\omega; \vec H} \|_{U^{D_*}_{q[A^{-3KD_*}M^{d_0}]}} + o(1).$$
From Proposition \ref{gu} we have
$$ \| f_{\{0\}^{D_*}; \vec H} \|_{U^{D_*}_{q[A^{-3KD_*}M^{d_0}]}} = \| \nu - 1 \|_{U^{D_*}_{q[A^{-3KD_*}M^{d_0}]}} = o(1)$$
while for $\omega \in \{0,1\}^{D_*} \backslash \{0\}^{D_*}$ we have $|f_{\omega;\vec H}| \leq 1$ and hence
$$ \| f_{\omega; \vec H} \|_{U^{D_*}_{q[A^{-3KD_*}M^{d_0}]}}  \leq 1.$$
The claim follows.

Now we turn to the $K'=1$ case.  This is effectively the same as the $K'=0$ case, except that $\nu-1$ is replaced by $(\nu-1) G_1$.  Thus, by repeating the above arguments, we reduce to showing that
$$ \| (\nu - 1) G_1 \|_{U^{D_*}_{q[A^{-3KD_*}M^{d_0}]}} = o(1).$$
Expanding out $G_1$ and using the triangle inequality for Gowers norms, it suffices to show that
$$
\E_{\vec h \in [M]^r} \| (\nu-1) {\mathcal D}^D_{P_1(\vec h)[-M,M],\dots,P_D(\vec h)[-M,M]}(\vec g_{1}) \|_{U^{D_*}_{q[A^{-3KD_*}M^{d_0}]}} = o(1)$$
so by H\"older's inequality it suffices to show that
$$
\E_{\vec h \in [M]^r} \| (\nu-1) {\mathcal D}^D_{P_1(\vec h)[-M,M],\dots,P_D(\vec h)[-M,M]}(\vec g_{1}) \|_{U^{D_*}_{q[A^{-3KD_*}M^{d_0}]}}^{2^{D_*}} = o(1).$$
We can expand the left-hand side as
$$
\E_{\vec h \in [M]^r} \langle g_{\omega,\vec h} \rangle_{\Box^{D_*+D}_{q[A^{-3KD_*}M^{d_0}], \dots, q[A^{-3KD_*}M^{d_0}], P_1(\vec h)[-M,M],\dots,P_D(\vec h)[-M,M]}}$$
where $q[A^{-3KD_*}M^{d_0}]$ appears  $D_*$ times in the Gowers norm, and
$$ g_{\omega, \{0\}^D} \coloneqq  \nu-1$$
and
$$ g_{\omega, \omega'} \coloneqq  g_{1,\omega'}$$
for all $\omega \in \{0,1\}^{D_*}$ and $\omega' \in \{0,1\}^D \backslash \{0\}^D$.  By the Gowers-Cauchy-Schwarz inequality, we can bound the above expression by
$$
\E_{\vec h \in [M]^r} \|\nu-1\|_{\Box^{D_*+D}_{q[A^{-3KD_*}M^{d_0}], \dots, q[A^{-3KD_*}M^{d_0}], P_1(\vec h)[-M,M],\dots,P_D(\vec h)[-M,M]}}^{2^{D_*}}$$
so by H\"older's inequality it suffices to show that
$$
\E_{\vec h \in [M]^r} \|\nu-1\|_{\Box^{D_*+D}_{q[A^{-3KD_*}M^{d_0}], \dots, q[A^{-3KD_*}M^{d_0}], P_1(\vec h)[-M,M],\dots,P_D(\vec h)[-M,M]}}^{2^{D_*+D}}=o(1).$$
But this follows from Proposition \ref{gu}.

Next, we turn to the $K'=2$ case.  Repeating the previous arguments, we reduce to showing that
$$ \| (\nu - 1) G_1 G_2 \|_{U^{D_*}_{q[A^{-3KD_*}M^{d_0}]}} = o(1).$$
But the product $G_1 G_2$ can be written as a dual function
$$ G_1 G_2 = {\mathcal D}^{2D}_{P_1(\vec h)[-M,M],\dots,P_D(\vec h)[-M,M],P_1(\vec h)[-M,M],\dots,P_D(\vec h)[-M,M]}(\vec g_{12})$$
where $\vec g_{12} = (g_{12,\omega})_{\omega \in \{0,1\}^{2D} \backslash \{0\}^{2D}}$ is defined by setting
$$ g_{12,(\omega, \{0\}^D)} \coloneqq  g_{1,\omega}$$
and
$$ g_{12,(\{0\}^D, \omega')} \coloneqq  g_{2,\omega'}$$
and
$$ g_{12,(\omega,\omega')} \coloneqq  1$$
for $\omega,\omega' \in \{0,1\}^D \backslash \{0\}^D$.  One can then repeat the $K'=1$ arguments (replacing $D$ with $2D$, and duplicating the polynomials $P_1,\dots,P_D$) to conclude this case.

Now we turn to the general case when $K'$ is allowed to be large.  Let $\eps>0$ be a parameter (which we will initially take to be independent of $N$) to be chosen later.  It will suffice to show that the left-hand side of \eqref{xfa} is $O_{K'}(\eps) + o(1)$ for each fixed $\eps>0$, since the claim then follows by using a diagonalisation argument to send $\eps$ slowly to zero.

Applying Theorem \ref{dfa} to each $G_{j'}$, we see that for any $j'=1,\dots,K'$, one can find an approximation $\tilde G_{j'} \colon \Z/N\Z \to [-2,2]$ to $G_{j'}$ with
$$ \| G_{j'} - \tilde G_{j'} \|_{L^2(\Z/N\Z)} \leq \eps$$
such that $\tilde G_{j'}$ has a representation of the form \eqref{fkcf} for some $k,k_i = O_{d,r,D,D_*,\eps}(1)$ (depending on $j$).  We then write
$$ G_{j'} = \tilde G_{j'} + E_{j'}$$
and
$$ G_1 \dots G_{K'} = \tilde G_1 \dots \tilde G_{K'} + E$$
for some error functions $E_{j'}, E$ with
\begin{equation}\label{dan}
 \|E_{j'} \|_{L^2(\Z/N\Z)} \leq \eps
\end{equation}
and
\begin{equation}\label{lak}
 E \ll_{K'} |E_1| + \dots + |E_{K'}|.
\end{equation}
We may split the left-hand side of \eqref{xfa} as the sum of
\begin{equation}\label{lok}
 \E_{x \in \Z/N\Z} (\nu(x)-1) F_1 \dots F_K(x) \tilde G_1 \dots \tilde G_{K'}(x)
\end{equation}
and
\begin{equation}\label{nad}
 \E_{x \in \Z/N\Z} (\nu(x)-1) F_1 \dots F_K(x) E(x).
\end{equation}
By \eqref{fkcf}, the expression \eqref{lok} is a bounded linear combination of $O_{d,r,D,D_*,\eps,K}(1)$ terms, each of which is $o(1)$ by the $K'=0$ case of this proposition (with $K$ replaced by various quantities of size $O_{d,r,D,D_*,\eps,K}(1)$).  Thus it suffices to show that the expression \eqref{nad} is $O_{K'}(\eps)+o(1)$.  Since the $F_j$ are bounded in magnitude by $1$, we can use \eqref{lak} bound \eqref{nad} in magnitude by
$$
\ll_{K'} \sum_{j'=1}^{K'} \E_{x \in \Z/N\Z} (\nu(x)+1) |E_{j'}(x)|.$$
By \eqref{noo} we have $ \E_{x \in \Z/N\Z} (\nu(x)+1) = 2+o(1)$, so by Cauchy-Schwarz it suffices to show that
$$
 \E_{x \in \Z/N\Z} (\nu(x)+1) E_{j'}(x)^2 \ll_{K'} \eps^2 + o(1)$$
for each $j'$.  From \eqref{dan} it suffices to show that
$$
 \E_{x \in \Z/N\Z} (\nu(x)-1) E_{j'}(x)^2 = o(1).$$
From the definition of $E$, we can expand the left-hand side as a bounded linear combination of the expressions
$$
 \E_{x \in \Z/N\Z} (\nu(x)-1) \tilde G_{j'}(x) \tilde G_{j'}(x),
$$
$$
 \E_{x \in \Z/N\Z} (\nu(x)-1) \tilde G_{j'}(x) G_{j'}(x),
$$
and
$$
 \E_{x \in \Z/N\Z} (\nu(x)-1) G_{j'}(x) G_{j'}(x).
$$
But these are all equal to $o(1)$ thanks to the $K'=0,1,2$ cases of this proposition respectively.  The claim follows.
\end{proof}

To use this proposition, we recall the \emph{dense model theorem}:

\begin{theorem}[Dense model theorem]  Let $\nu \colon \Z/N\Z \to \R^+$ be a function with $\E_{x \in \Z/N\Z} \nu(x) = 1+o(1)$, and let ${\mathcal F}$ be a collection of functions $g \colon \Z/N\Z \to \R^+$ bounded in magnitude by $1+o(1)$.  Suppose that for any fixed $k$ (independent of $N$) and any $g_1,\dots,g_k \in {\mathcal F}$, one has
$$ \E_{x \in \Z/N\Z} g_1(x) \dots g_k(x) (\nu(x)-1) = o(1)$$
as $N \to \infty$ uniformly in the choice of $f_1,\dots,f_k$.
Then for any function $f \colon \Z/N\Z \to \R$ with $0 \leq g(x) \leq \nu(x)$ for all $x$, there exists a function $f' \colon \Z/N\Z \to \R$ with $0 \leq f'(x) \leq 1+o(1)$ for all $x$ such that
$$ \E_{x \in \Z/N\Z} (f(x)-f'(x))g(x) = o(1)$$
as $N \to\infty$ uniformly for all $g \in {\mathcal F}$.
\end{theorem}

\begin{proof} See \cite[Theorem 1.1]{rttv}, \cite[Theorem 7.1]{tz-pattern}, or \cite[Theorem 4.8]{gow}; the formulation here is closest to that in \cite{rttv}.  (This theorem also appeared implicitly in \cite{gt-primes}.)
\end{proof}

\begin{proposition}[A dense model for $f$]\label{dense-model}
Let the notation and hypotheses be as in Theorem \ref{l1}.  Then there exists a function $f' \colon \Z/N\Z \to \R$ with the pointwise bounds $0 \leq f' \ll_\kappa 1$ such that
$$ \E_{x \in \Z/N\Z} (f-f')(x) {\mathcal D}^{D_*}_{q[A^{-2D_*}M^{d_0}]}( \vec g )(x) = o(1)$$
and
$$\E_{x \in \Z/N\Z} (f-f')(x)  \E_{\vec h \in [M]^r} {\mathcal D}^D_{P_1(\vec h)[-M,M],\dots,P_D(\vec h)[-M,M]}(\vec g')(x) = o(1)$$
for any $q \leq A^{D_*}$ and any tuples $\vec g = (g_\omega)_{\omega \in \{0,1\}^{D_*} \backslash \{0\}^{D_*}}$ and $\vec g' = (g'_\omega)_{\omega \in \{0,1\}^D \backslash \{0\}^D}$ of functions $g_\omega, g'_\omega \colon \Z/N\Z \to [-1,1]$.
\end{proposition}

\begin{proof}
If we let $c>0$ be a sufficiently small quantity depending on $\kappa$, then by \eqref{kap} we can bound $c f$ pointwise in magnitude by $\nu$.  Applying the dense model theorem with ${\mathcal F}$ consisting of those functions of the form
$$ 
{\mathcal D}^{D_*}_{q[A^{-2D_*}M^{d_0}]}( \vec g ) $$
or
$$ \E_{\vec h \in [M]^r} {\mathcal D}^D_{P_1(\vec h)[-M,M],\dots,P_D(\vec h)[-M,M]}(\vec g')$$
for any $q \leq A^{D_*}$ and any tuples $\vec g = (g_\omega)_{\omega \in \{0,1\}^{D_*} \backslash \{0\}^{D_*}}$ and $\vec g' = (g'_\omega)_{\omega \in \{0,1\}^D \backslash \{0\}^D}$ of functions $g_\omega, g'_\omega \colon \Z/N\Z \to [-1,1]$, and using Proposition \ref{vari}, we obtain the claim with $f$ replaced by $cf$.  Dividing by $c$, the proposition follows.
\end{proof}

The above properties of $f-f'$ are not directly useful in applications because of the requirement that the functions $g_\omega, g'_\omega$ take values in $[-1,1]$.  However, we can use \emph{densification} argument of Conlon, Fox, and Zhao \cite{cfz} to relax this hypothesis, to a hypothesis that each $g_\omega$ or $g'_\omega$ is bounded by either $1$ or $\nu$:

\begin{proposition}[Densification]  Let the notation and hypotheses be as in Theorem \ref{l1}.  Let $f'$ be the function in Proposition \ref{dense-model}.  Then one has
\begin{equation}\label{sax}
 \E_{x \in \Z/N\Z} (f-f')(x) {\mathcal D}^{D_*}_{q[A^{-2D_*}M^{d_0}]}( \vec g )(x) = o(1)
\end{equation}
and
\begin{equation}\label{xas}
\E_{x \in \Z/N\Z} (f-f')(x)  \E_{\vec h \in [M]^r} {\mathcal D}^D_{P_1(\vec h)[-M,M],\dots,P_D(\vec h)[-M,M]}(\vec g')(x) = o(1)
\end{equation}
for any $q \leq A^{D_*}$ and any tuples $\vec g = (g_\omega)_{\omega \in \{0,1\}^{D_*} \backslash \{0\}^{D_*}}$ and $\vec g' = (g'_\omega)_{\omega \in \{0,1\}^D \backslash \{0\}^D}$ of functions $g_\omega, g'_\omega \colon \Z/N\Z \to \R$, each of which are bounded in magnitude either by $1$ or by $\nu$.  (Thus, for instance, for $\omega \in \{0,1\}^{D_*}$, one either has $|g_\omega(x)| \leq 1$ for all $x \in \Z/N\Z$, or $|g_\omega(x)| \leq \nu(x)$ for all $x \in \Z/N\Z$.)
\end{proposition}

\begin{proof}  We just establish the claim \eqref{xas}, as the claim \eqref{sax} is proven by an essentially identical argument.

Let $J$ denote the number of indices $\omega' \in \{0,1\}^D \backslash \{0\}^D$ for which one is assuming the bound $|g'_{\omega}(x)| \leq \nu(x)$ (rather than $|g'_{\omega}(x)| \leq 1$), thus $0 \leq J \leq 2^D-1$.  We prove the claim by induction on $J$.  For $J=0$, the claim is immediate from Proposition \ref{dense-model}.  Now suppose inductively that $J \geq 1$, and that the claim has already been proven for $J=1$.

By hypothesis, we have a tuple $\omega_0$ in $\{0,1\}^D \backslash \{0\}^D$ such that $g'_{\omega_0}$ is bounded in magnitude by $\nu$, with $g_\omega$ being bounded by $\nu$ for $J-1$ of the other indices $\{0,1\}^D \backslash \{0\}^D$ and by $1$ for all remaining indices.  We also adopt the convention $g'_{\{0\}^D} \coloneqq  f-f'$, so that we may write the left-hand side of \eqref{xas} as
$$ \E_{\vec h \in [M]^r} \E_{x \in \Z/N\Z} \E_{\vec k \in \prod_{i=1}^D P_i(\vec h)[-M,M]} \prod_{\omega \in \{0,1\}^D} g'_{\omega}( x + \omega \cdot \vec k ).$$
Making the change of variables $x \mapsto x - \omega_0 \cdot \vec k$, we may write this as
$$ \E_{x \in \Z/N\Z} g_{\omega_0}(x) F(x)$$
where
$$ F(x) \coloneqq  \E_{\vec h \in [M]^r} \E_{\vec k \in \prod_{i=1}^D P_i(\vec h)[-M,M]} \prod_{\omega \in \{0,1\}^D \backslash \{\omega_0\}} g'_{\omega}( x + (\omega-\omega_0) \cdot \vec k ).$$
Since $g_{\omega_0}$ is bounded in magnitude by $\nu$, we see from \eqref{noo} and Cauchy-Schwarz that it suffices to show that
$$ \E_{x \in \Z/N\Z} \nu(x) F(x)^2 = o(1).$$
We split this into two claims
\begin{equation}\label{cl-1}
\E_{x \in \Z/N\Z} (\nu(x)-1) F(x)^2 = o(1)
\end{equation}
and
\begin{equation}\label{cl-2}
\E_{x \in \Z/N\Z} F(x)^2 = o(1).
\end{equation}
We begin with \eqref{cl-1}.  We can expand the left-hand side as
\begin{align*} & \E_{\vec h, \vec h' \in [M]^r} \E_{x \in \Z/N\Z} \E_{\vec k \in \prod_{i=1}^D P_i(\vec h)[-M,M] \times \prod_{i=1}^D P_i(\vec h')[-M,M]} 
(\nu(x)-1) \\
&\quad \prod_{\omega \in \{0,1\}^D \backslash \{\omega_0\}} g'_{\omega}( x + (\omega-\omega_0,0) \cdot \vec k ) g'_{\omega}( x + (0,\omega-\omega_0) \cdot \vec k ).
\end{align*}
By the Gowers-Cauchy-Schwarz inequality, we can bound this by
\begin{align*}
&\E_{\vec h, \vec h' \in [M]^r} \| \nu(x)-1\|_{\Box^{2D}_{P_1(\vec h)[-M,M],\dots,P_D(\vec h)[-M,M],P_1(\vec h')[-M,M],\dots,P_D(\vec h')[-M,M]}} \\
&\quad
\prod_{\omega \in \{0,1\}^D \backslash \{\omega_0\}} \| g'_\omega \|_{\Box^{2D}_{P_1(\vec h)[-M,M],\dots,P_D(\vec h)[-M,M],P_1(\vec h')[-M,M],\dots,P_D(\vec h')[-M,M]}}.
\end{align*}
Since $g'_\omega$ is bounded in magnitude by $\nu+1 = (\nu-1)+2$, it thus suffices by H\"older's inequality to show that
$$
\E_{\vec h, \vec h' \in [M]^r} \| \nu-1\|_{\Box^{2D}_{P_1(\vec h)[-M,M],\dots,P_D(\vec h)[-M,M],P_1(\vec h')[-M,M],\dots,P_D(\vec h')[-M,M]}}^{2^{2D}} = o(1).$$
But this follows from Proposition \ref{gu}.

Now we show \eqref{cl-2}.  By H\"older's inequality, it suffices to show that
\begin{equation}\label{cl-4}
\E_{x \in \Z/N\Z} F(x)^4 \ll_{r,D} 1 + o(1)
\end{equation}
and
\begin{equation}\label{cl-abs}
\E_{x \in \Z/N\Z} |F(x)| = o(1).
\end{equation}
To prove \eqref{cl-4}, we bound $g'_\omega$ by $\nu+1$.  The estimate then follows from the following claim, proven in Appendix \ref{polyapp}:

\begin{lemma}\label{ps2}  
We have
$$ \E_{x \in \Z/N\Z} \left(\E_{\vec h \in [M]^r} \E_{\vec k \in \prod_{i=1}^D P_i(\vec h)[-M,M]} \prod_{\omega \in \{0,1\}^D \backslash \{\omega_0\}} \left(1+\nu( x + (\omega-\omega_0) \cdot \vec k )\right)\right)^4 \ll_{r,D} 1 + o(1).$$
\end{lemma}

Finally, we show \eqref{cl-abs}.  It suffices to show that
$$\E_{x \in \Z/N\Z} g(x) F(x) = o(1)$$
whenever $g \colon \Z/N\Z \to [-1,1]$ is a function.  But this follows from the induction hypothesis, since the left-hand side is simply \eqref{xas} with $g'_{\omega_0}$ replaced with $g$.  This proves \eqref{xas}; the estimate \eqref{sax} is proven similarly and is left to the reader.
\end{proof}

\begin{corollary}\label{Coro} Let the notation and hypotheses be as in Theorem \ref{l1}.  Let $f'$ be the function in Proposition \ref{dense-model}.  Then one has
\begin{equation}\label{sax-1}
 \| f - f' \|_{U^{D_*}_{q[A^{-2D_*}M^{d_0}]}} = o(1)
\end{equation}
for all $q \leq A^{D_*}$, and also
\begin{equation}\label{xas-1}
 \E_{\vec h \in [M]^r} \| f - f' \|_{\Box^D_{P_1(\vec h)[-M,M], \dots, P_D(\vec h)[-M,M]}} = o(1).
\end{equation}
\end{corollary}

\begin{proof}
By linearity, the bounds \eqref{sax}, \eqref{xas} continue to hold if the hypotheses that $|g_\omega|, |g'_\omega|$ are bounded by $1$ or by $\nu$ are replaced by $|g_\omega|, |g'_\omega| \ll_\kappa \nu+1$.  In particular, we can set all of the $g_\omega, g'_\omega$ equal to $\Lambda'_{b,W}-f$.  The bound \eqref{sax} then gives \eqref{sax-1}, while \eqref{xas} gives
$$
 \E_{\vec h \in [M]^r} \| f - f' \|_{\Box^D_{P_1(\vec h)[-M,M], \dots, P_D(\vec h)[-M,M]}}^{2^D} = o(1)
$$
which implies \eqref{xas-1} by H\"older's ineqality.
\end{proof}

The bound \eqref{xas-1} already gives the first claim of Theorem \ref{l1}.  It remains to prove the second claim. Thus we may now assume that $d_0=d$, $M \geq (A^{-1} N)^{1/d}$, and $f = \Lambda'_{b,W}$.

We now invoke the following estimates on global uniformity norms of the von Mangoldt function, arising from the results in \cite{gt-primes} (using the inverse Gowers and Mobius-nilsequences conjectures proven in \cite{gtz-uk}, \cite{gt-mobius} respectively):

\begin{proposition}[Global Gowers uniformity]\label{l0}  Let the notation and hypotheses be as in Theorem \ref{l1}, with $d_0=d$, $M \geq (A^{-1} N)^{1/d}$, and $f = \Lambda'_{b,W}$.  Then one has
$$ \| \Lambda'_{b,W} - 1 \|_{U^{D_*}_{q[A^{-2D_*}M^{d_0}]}} = o(1)$$
for all $q \leq A^{D_*}$.
\end{proposition}

\begin{proof}  Raising both sides to the power $2^{D_*}$ and expanding, it suffices to show that
$$ \E_{\vec h \in [A^{-2D_*}M^{d_0}]^{D_*}} \E_{x \in \Z/N\Z} \prod_{\omega \in I}  \Lambda'_{b,W}( x + q \omega \cdot \vec h) = 1 + o(1) $$
for all $I \subset \{0,1\}^{D_*}$.   The contribution of those $x$ for which $N - q D_* A^{-2D_*} M^{d_0} \leq x \leq N$ can be easily verified by standard upper bound sieves to be $o(1)$; replacing $\theta$ in \eqref{bawn} by the von Mangoldt function and removing the restriction $x \geq R$ also contributes an error of $o(1)$.  Thus it suffices to show that
\begin{equation}\label{wow}
 \E_{\vec h \in [A^{-2D_*}M^{d_0}]^{D_*}} \E_{x \in [N]} \prod_{\omega \in I} \frac{\phi(W)}{W} \Lambda( W(x + q \omega \cdot \vec h) + b) = 1 + o(1) 
\end{equation}
Suppose first that $w$ (and hence $W$ and $A$) is a fixed quantity independent of $N$.  Since 
$$ A^{-2D_*-1} N \leq A^{-2D_*} M^{d_0} \leq A^{-2D_*+1} N $$
we see that the quantity $A^{-2D_*} M^{d_0}$ is now comparable to $N$.  The expression \eqref{wow} is now of a form that can be handled by the results in \cite{gt-primes} (note that a modification of Example 2 from that paper shows that the linear forms here have finite complexity).  Using the results of that paper (and of \cite{gtz-uk}, \cite{gt-mobius}), we see that the left-hand side of \eqref{wow} is equal to
$$ (1+O(1/A)) \left(\frac{\phi(W)}{W}\right)^{|I|} \prod_p \beta_p + o(1)$$
where for each prime $p$, $\beta_p$ is the quantity
$$ \beta_p \coloneqq  \E_{\vec h \in (\Z/p\Z)^{D_*}} \E_{x \in \Z/p\Z} \prod_{\omega \in I} \frac{p}{p-1} 1_{p \nmid W(x + q \omega \cdot \vec h) + b}.$$
For $p \leq w$, the quantity $\beta_p$ simplifies to $(\frac{p}{p-1})^{|I|}$, and hence the left-hand side of \eqref{wow} simplifies to
$$ (1+O(1/A)) \prod_{p>w} \beta_p + o(1).$$
For $p>w$ not dividing $q$, the linear forms $(x,\vec h) \mapsto W(x + q \omega \cdot \vec h)$ are not scalar multiples of each other over $\Z/p\Z$, and one can then easily verify that $\beta_p = 1 + O_{D_*}(1/p^2) = \exp( O_{D_*}(1/p^2) )$ in these cases, leading to a net multiplicative contribution of $\exp( O_{D_*}( 1/w ) )$ to the product $\prod_{p>w} \beta_p$.  For $p>w$ dividing $q$, we can crudely estimate $\beta_p$ as $1 + O_{D_*}(1/p) = \exp( O_{D_*}(1/w) )$; since $q \leq A^{D_*} = W^{D_*/\kappa} = \exp( O_{D_*,\kappa}( w ) )$, the number of such primes $p$ is at most $O_{D_*,\kappa}(w / \log w)$.  We conclude that the net contribution of these primes to $\prod_{p>w} \beta_p$ is $\exp( O_{D_*,\kappa}(1 / \log w) )$.  We conclude that the left-hand side of \eqref{wow} is
$$ (1+O(1/A)) \exp( O_{D_*,\kappa}(1/\log w) ) + o(1) $$
for fixed $w$.  Letting $w$ grow sufficiently slowly to infinity, we obtain the claim.
\end{proof}

Combining Proposition \ref{l0} with Corollary \ref{Coro} and Theorem \ref{avg-norm}, we can now finally establish the second part of Theorem \ref{l1}.  Let $f'$ be the function in Proposition \ref{dense-model}.  From \eqref{sax-1}, Proposition \ref{l0}, and the triangle inequality we have
$$ \|f'-1\|_{U^{D_*}_{q[A^{-2D_*}M^{d_0}]}} = o(1)$$
for all $q \leq A^{D_*}$. Since $f'-1 = O_\kappa(1)$, we can divide by a constant depending only on $\kappa$ and invoke Theorem \ref{avg-norm} to conclude that
$$ \E_{\vec h \in [M]^r} \| f'-1 \|_{\Box^D_{P_1(\vec h)[-M,M], \dots, P_D(\vec h)[-M,M]}} = o(1).$$
Combining this with \eqref{xas-1} and the triangle inequality, we obtain the second claim of Theorem \ref{l1}.

\section{Applying a generalized von Neumann inequality}

We continue to use the parameters $d,r,d_0,D,\kappa,W,A,R,N$ from the previous section.
The objective of this section is to establish the following bound:

\begin{theorem}\label{l2}  For each $i=1,\dots,D$, let $b_i \in [W]$ be coprime to $W$.  Let $M_0$ be a quantity with $\log^{1/\kappa} N \leq M_0 \leq A N^{1/d}$, and set $M \coloneqq  A^{-1} M_0$.
For $i=1,\dots,D$, we set $f_i$ to be a function with the pointwise bounds $|f_i| \ll \Lambda'_{b_i,W}+1$.  For each $i=1,\dots,d$, let $R_i \in \Z[\m_1,\dots,\m_r]$ be a polynomial of degree at most $d$, with all coefficients bounded in magnitude by $A$.  Assume also that for $i=2,\dots,D$, $R_i - R_1$ has degree at least $d_0$.  Then
$$ \E_{x \in \Z/N\Z} \E_{\vec m \in [M_0]^r} \prod_{i=1}^D f_i( x + R_i( \vec m ) ) \ll \|f_1\|^c + o(1)$$
for some $1 \ll_{d,r,D} c \ll_{d,r,D} 1$, where $\|f_1\|$ is short for a norm of the form
\begin{equation}\label{sof}
\|f_1\| \coloneqq  \E_{\vec h \in [M]^r} \| f_1 \|_{\Box^{D'}_{P_1(\vec h)[-M,M], \dots, P_{D'}(\vec h)[-M,M]}} 
\end{equation}
for some natural number $D' = O_{d,r,D}(1)$, and some polynomials $P_1,\dots,P_{D'} \in \Z[\h_1,\dots,\h_r]$ of degree between $d_0-1$ and $d-1$ and coefficients $O( A^{O_{d,r,D}(1)} )$.
\end{theorem}

The arguments in \cite[\S 5]{tz-pattern} (based on van der Corput's method, the Cauchy-Schwarz inequality, PET induction, and a ``polynomial forms'' condition on the pseudorandom majorant $\nu$) ``morally'' permit one to deduce Theorem \ref{l2} from Theorem \ref{l1}.  However, the setup here differs from that in \cite{tz-pattern} in several minor technical aspects, most notably the multidimensional nature of the parameter $\m$, the non-constancy of the $b_i$ parameter in $i$, and the much larger value of the scale parameter $M_0$.  As such, we will need to adapt the argument from \cite{tz-pattern} to the current setting.

For inductive purposes it is convenient to prove a more general form of Theorem \ref{l2}. 
We need the following definitions (inspired by, though not absolutely identical to, analogous definitions in \cite{tz-pattern}):

\begin{definition}[Polynomial system]  A \emph{polynomial system} ${\mathcal S}$ consists of the following objects:
\begin{itemize}
\item An integer $D_{\mathcal S} \geq 0$, which we will call the \emph{number of fine degrees of freedom};
\item A non-empty finite index set ${\mathcal A}$ (the elements of which we will call the \emph{nodes} of the system);
\item A polynomial $R_\alpha \in \Z[\m_1,\dots,\m_r,\h_1,\dots,\h_{D_{\mathcal S}}]$ of degree at most $d$ attached to each node $\alpha \in {\mathcal A}$;
\item A distinguished node $\alpha_0 \in {\mathcal A}$;
\item A (possibly empty) collection ${\mathcal A}' \subset {\mathcal A} \backslash \{\alpha_0\}$ of \emph{inactive nodes}.	The nodes in ${\mathcal A} \backslash {\mathcal A}'$ will be referred to as \emph{active}, thus for instance the distinguished node $\alpha_0$ is always active.
\end{itemize}
We say that a node $\alpha$ is \emph{linear} if $R_\alpha - R_{\alpha_0}$ is at most linear in $\m_1,\dots,\m_r$ (i.e. it has degree at most $1$ when viewed as a polynomial in $\m_1,\dots,\m_r$ with coefficients in $\Z[\h_1,\dots,\h_{D_{\mathcal S}}]$), thus for instance the distinguished mode $\alpha_0$ is always linear.  We say that ${\mathcal S}$ is linear if all active nodes are linear.  We require polynomial systems to obey three additional axioms:
\begin{itemize}
\item If $\alpha,\beta$ are distinct nodes in ${\mathcal A}$, then $R_\alpha - R_{\beta}$ is not constant in $\m_1,\dots,\m_r,\h_1,\dots,\h_{D_{\mathcal S}}$ (i.e. it does not lie in $\Z$).
\item If $\alpha \in {\mathcal A} \backslash \{\alpha_0\}$, then $R_\alpha - R_{\alpha_0}$ has degree at least $d_0$ in $\m_1,\dots,\m_r,\h_1,\dots,\h_{D_{\mathcal S}}$.
\item If $\alpha,\beta$ are distinct \emph{linear} nodes in ${\mathcal A}$, then $R_\alpha - R_\beta$ is not constant in $\m_1,\dots,\m_r$ (i.e. it does not lie in $\Z[\h_1,\dots,\h_{D_{\mathcal S}}]$).
\end{itemize}
\end{definition}

\begin{definition}[Realizations and averages]  Let ${\mathcal S}$ be a polynomial system.  A \emph{realization} $\vec f = (f_\alpha)_{\alpha \in {\mathcal A}}$ of ${\mathcal S}$ is an assignment of functions $f_\alpha \colon \Z/N\Z \to \R$ to each node $\alpha$ with the following properties:
\begin{itemize}
\item For any node $\alpha$, one has the pointwise bound $|f_\alpha| \ll \nu_{b_\alpha}+1$ for some $b_\alpha \in [W]$ coprime to $W$.
\item For any inactive node $\alpha$, one has $f_\alpha = \nu_{b_\alpha}+1$ for some $b_\alpha \in [W]$ coprime to $W$.
\end{itemize}
We define the \emph{average} $\Lambda_{\vec S}(\vec f)$ to be the quantity
\begin{equation}\label{lvs-def}
\Lambda_{\vec S}(\vec f) \coloneqq  \E_{x \in \Z/N\Z} \E_{\vec m \in [M_0]^r} \E_{\vec h \in [M]^{D_{\mathcal S}}} \prod_{\alpha \in {\mathcal A}} f_\alpha( x + R_\alpha( \vec m, \vec h ) ).
\end{equation}
\end{definition}

Theorem \ref{l2} is then a special case of the following more general statement.

\begin{theorem}\label{l3} Let $C_0$ be a quantity depending only on $d,r,D$, and assume $\kappa$ sufficiently large depending on $C_0$.  Let $C_1$ be a quantity depending only on $d,r,D,C_0,\kappa$ (in particular, $C_0,C_1$ are independent of $N$).  Let ${\mathcal S}$ be a system with at most $C_0$ nodes and at most $C_0$ fine degrees of freedom, with all polynomials $R_\alpha$ associated to the system having coefficients bounded in magnitude by $A^{C_1}$.  Let $\vec f$ be a realization of ${\mathcal S}$.  Then
\begin{equation}\label{lsf}
 \Lambda_{\vec S}(\vec f) \ll \|f_{\alpha_0}\|^c + o(1)
\end{equation}
for some $c>0$ depending on $d,r,D,C_0,C_1$, and $\|f_1\|$ is a norm of the form \eqref{sof} with $D' = O_{d,r,D,C_0,C_1}(1)$, and $P_1,\dots,P_{D'} \in \Z[\h_1,\dots,\h_r]$ of degree between $d_0-1$ and $d-1$ with coefficients $A^{O_{d,r,D,C_0,C_1}(1)}$.
\end{theorem}

Indeed, Theorem \ref{l2} is the special case in which the system ${\mathcal S}$ consists of the nodes ${\mathcal A} = \{1,\dots,D\}$ with distinguished node $\alpha_0 = 1$ and all nodes active, with $D_{\mathcal S}=0$, and $R_i$ and $b_i$ as indicated by Theorem \ref{l2}.

It remains to establish Theorem \ref{l3}.  This will follow the same three-step procedure used in \cite{tz-pattern}.

\subsection{Reduction to the linear case}

The first step is to use the van der Corput method and PET induction to reduce matters to the linear case.  We need some further definitions, again essentially from \cite{tz-pattern}.

Given two nodes $\alpha,\beta$ in a polynomial system ${\mathcal S}$, we define the \emph{distance} $d(\alpha,\beta)$ between the nodes to be the degree in $\m_1,\dots,\m_r$ of the polynomial $R_\alpha - R_\beta$.  This distance is symmetric, reflexive, and obeys the ultrametric triangle inequality
$$ d(\alpha,\gamma) \leq \max( d(\alpha,\beta), d(\beta,\gamma))$$
for all nodes $\alpha,\beta,\gamma$.  We define the \emph{diameter} of the system to be the maximal value of $d(\alpha,\beta)$ for $\alpha,\beta$ ranging over active nodes, and define an \emph{extreme} node to be an active node $\alpha$ such that $d(\alpha,\alpha_0)$ is equal to the diameter of ${\mathcal S}$; note from the ultrametric triangle inequality that there is always at least one such node.

Given a node $\alpha$, we then call two nodes $\beta,\gamma$ \emph{equivalent} relative to $\alpha$ if $d(\beta,\gamma) < d(\beta,\alpha)$; this is an equivalence relation on nodes, and every equivalence class has a well-defined distance to $\alpha$.  We then define the \emph{weight} $\vec w_\alpha({\mathcal S})$ of ${\mathcal S}$ relative to $\alpha$ to be the vector $(w_1,\dots,w_d) \in \Z_+^d$, where $w_i$ is the number of equivalence classes relative to $\alpha$ at a distance $i$ from $\alpha$.  Thus for instance ${\mathcal S}$ will be linear if and only if the weight $\vec w_\alpha({\mathcal S})$ relative to a node $\alpha$ takes the form $(w_1,0,\dots,0)$.  We order weights lexicographically, thus $(w_1,\dots,w_d) < (w'_1,\dots,w'_d)$ if there is $1 \leq i \leq d$ such that $w_i < w'_i$ and $w_j = w'_j$ for all $i < j \leq d$.

For a given choice of constants $C_0,C_1$ and a weight vector $\vec w$, let $P(C_0,C_1,\vec w)$ denote the assertion that Theorem \ref{l3} holds for the given choice of $C_0,C_1$ and for all polynomial systems of weight $\vec w$ relative to some extreme\footnote{The requirement that $\alpha$ be extreme was mistakenly omitted in our previous paper \cite{tz-pattern}; it is needed to force the PET induction to terminate at a linear system.} node $\alpha$.  The key inductive claim is then

\begin{proposition}[PET induction step]\label{induct}  For any $C_0,C_1,\vec w = (w_1,\dots,w_d)$ which is nonlinear in the sense that $w_i \neq 0$ for some $i=2,\dots,d$, there exist a finite collection of triples $(C'_0,C'_1,\vec w')$ with $\vec w' < \vec w$, such that if $P(C'_0,C'_1,\vec w')$ holds for all triples in this collection, then $P(C_0,C_1,\vec w)$ holds.
\end{proposition}

Since the number of weight vectors $\vec w$ that can arise from systems ${\mathcal S}$ of at most $C_0$ nodes is finite, and the collection of all weight vectors is well-ordered,  we conclude from Proposition \ref{induct} that if $P(C_0,C_1,\vec w)$ holds for all linear $\vec w$, then it holds for all $\vec w$.  This implies that to prove Theorem \ref{l3}, it suffices to do so in the case when ${\mathcal S}$ is linear.

We now establish Proposition \ref{induct}.  Let ${\mathcal S}$ be a polynomial system with at most $C_0$ nodes and at most $C_0$ fine degrees of freedom, and of weight $\vec w$ relative to some extreme node $\alpha$, and with all polynomials having coefficients bounded in magnitude by $A^{C_1}$.  Since $\vec w$ is nonlinear, the diameter of ${\mathcal S}$ is at least two.  By subtracting $R_\alpha$ from each of the other $R_\beta$ (noting that this does not affect the metric $d$ or the average $\Lambda_{\vec S}(\vec f)$, we may assume that $R_\alpha = 0$ (at the cost of increasing the coefficient bound from $A^{C_1}$ to $2A^{C_1}$).  As $\alpha$ was extreme, we now note that the polynomial $R_{\alpha_0}$ associated to $\alpha_0$ has maximal $\m$-degree amongst all the polynomials associated to active nodes.

We split ${\mathcal A} = {\mathcal A}_0 \cup {\mathcal A}_1$, where ${\mathcal A}_0$ is the set of nodes $\beta$ with $d(\alpha,\beta)=0$, and ${\mathcal A}_1$ is the set of nodes $\beta$ with $d(\alpha,\beta) \geq 1$; note that the distinguished node $\alpha_0$ lies in ${\mathcal A}_1$.  We can then factor
$$ \Lambda_{\vec S}(\vec f) = \E_{\vec h \in [M]^{D_{\mathcal S}}} \E_{x \in \Z/N\Z} F_{\vec h}(x) \E_{\vec m \in [M_0]^r} G_{\vec m, \vec h}(x) $$
where
$$ F_{\vec h}(x) \coloneqq  \prod_{\beta \in {\mathcal A}_0} f_\beta( x + R_\beta( 0, \vec h ) ) $$
and
$$ G_{\vec m, \vec h}(x) \coloneqq  \prod_{\beta \in {\mathcal A}_1} f_\beta( x + R_\beta( \vec m, \vec h ) ).$$
By hypothesis, each $f_\beta$ is bounded in magnitude by $\nu_{b_\beta}+1$ for some $b_\beta \in [W]$ coprime to $W$.  Thus we have the pointwise bound
$$ | F_{\vec h}(x)| \leq H_{\vec h}(x)$$
where
$$ H_{\vec h}(x) \coloneqq  \prod_{\beta \in {\mathcal A}_0} (\nu_{b_\beta}+1)( x + R_\beta( 0, \vec h ) ).$$
 We can pointwise bound $|G_{\vec m,\vec h}(x)| \leq K_{\vec m, \vec h}(x)$ where
$$ K_{\vec m, \vec h}(x) \coloneqq  \prod_{\beta \in {\mathcal A}_1} (\nu_\beta + 1)( x + R_\beta( \vec m, \vec h ) ).$$

In Appendix \ref{polyapp} we will establish the following bounds:

\begin{lemma}\label{st-lemma}  With the notation as above, we have
\begin{equation}\label{st1}
\E_{\vec h \in [M]^{D_{\mathcal S}}} \E_{x \in \Z/N\Z} H_{\vec h}(x)  = 2^{|{\mathcal A}_0|} + o(1).
\end{equation}
and
\begin{equation}\label{st2}
\E_{\vec h \in [M]^{D_{\mathcal S}}} \E_{x \in \Z/N\Z} H_{\vec h}(x) (\E_{\vec m \in [M]^r + \vec a} K_{\vec m, \vec h}(x))^2 = 2^{|{\mathcal A}_0|+2|{\mathcal A}_1|} + o(1)
\end{equation}
uniformly for all $\vec a \in \Z^r$.  
\end{lemma}

By \eqref{st1} and the Cauchy-Schwarz inequality, we see that to show \eqref{lsf}, it suffices to show that
\begin{equation}\label{pre-vdc}
 \E_{\vec h \in [M]^{D_{\mathcal S}}} \E_{x \in \Z/N\Z} H_{\vec h}(x) (\E_{\vec m \in [M_0]^r} G_{\vec m, \vec h}(x))^2 \ll \|f_{\alpha_0}\|^{2c} + o(1).
\end{equation}

We now apply the van der Corput method.  We can pointwise bound $|G_{\vec m,\vec h}(x)| \leq K_{\vec m, \vec h}(x)$ where
$$ K_{\vec m, \vec h}(x) \coloneqq  \prod_{\beta \in {\mathcal A}_1} (\nu_\beta + 1)( x + R_\beta( \vec m, \vec h ) ).$$
By covering the boundary of $[M_0]^r$ by about $O_r(A^{r-1})$ translates of $[M]^r$, we see that there is a collection $\Sigma$ of $O_r(A^{r-1})$ elements $\vec a$ of $\Z^r$ such that
$$ \left|\E_{\vec m \in [M_0]^r + [M]^r} G_{\vec m, \vec h}(x) - \E_{\vec m \in [M_0]^r} G_{\vec m, \vec h}(x)\right| \ll_r A^{-r} \sum_{\vec a \in \Sigma} \E_{\vec m \in [M]^r + \vec a} K_{\vec m, \vec h}(x) $$
and hence by \eqref{st2} and the triangle inequality (and noting that $A^{-r} \times A^{r-1} = o(1)$)
$$ \E_{\vec h \in [M]^{D_{\mathcal S}}} \E_{x \in \Z/N\Z} H_{\vec h}(x) \left(\E_{\vec m \in [M_0]^r + [M]^r} G_{\vec m, \vec h}(x) - \E_{\vec m \in [M_0]^r} G_{\vec m, \vec h}(x)\right)^2 = o(1)$$
for any $h_1 \in [M]^r$.  By the triangle inequality, to prove \eqref{pre-vdc} it suffices to show that
$$
 \E_{\vec h \in [M]^{D_{\mathcal S}}} \E_{x \in \Z/N\Z} H_{\vec h}(x) (\E_{\vec m \in [M_0]^r + [M]^r} G_{\vec m, \vec h}(x))^2 \ll \|f_{\alpha_0}\|^{2c} + o(1).
$$
By Cauchy-Schwarz, we can bound the left-hand side by
$$
 \E_{\vec h \in [M]^{D_{\mathcal S}}} \E_{x \in \Z/N\Z} H_{\vec h}(x) \E_{\vec m \in [M_0]^r} (\E_{\vec h' \in [M]^r} G_{\vec m + \vec h', \vec h}(x))^2
$$
which we may expand as
$$
 \E_{\vec h \in [M]^{D_{\mathcal S}}} \E_{h', h'' \in [M]^r} \E_{x \in \Z/N\Z} \E_{\vec m \in [M_0]^r}  H_{\vec h}(x) G_{\vec m + \vec h', \vec h}(x) G_{\vec m+\vec h'',\vec h}(x).
$$
Comparing this with \eqref{lvs-def}, we see that this expression can be written as
$$ \Lambda_{{\mathcal S}'}(\vec f')$$
where the polynomial system ${\mathcal S}'$ and the realization $\vec f' = (f'_\beta)_{\beta \in {\mathcal A}'}$ are defined as follows.

\begin{itemize}
\item The number of fine degrees of freedom is $D_{{\mathcal S}'} \coloneqq  D_{\mathcal S} + 2r$;
\item The set of nodes ${\mathcal A}'$ consists of the disjoint union of ${\mathcal A}_0$, ${\mathcal A}_1$, and another copy ${\mathcal A}'_1$ of ${\mathcal A}_1$;
\item The polynomials
$$R'_\beta \in \Z[\m_1,\dots,\m_r.\h_1,\dots,\h_{D_{\mathcal S}}, \h'_1,\dots,\h'_r,\h''_1,\dots,\h''_r] = \Z[ \vec \m, \vec \h, \vec \h', \vec \h'']$$
for $\beta \in {\mathcal A}' = {\mathcal A}_0 \cup {\mathcal A}_1 \cup {\mathcal A}'_1$ are defined by setting
$$ R'_\beta \coloneqq  R_\beta(0, \vec \h)$$
for $\beta \in {\mathcal A}_0$,
$$ R'_\beta \coloneqq  R_\beta(\vec \m + \vec \h', \vec \h)$$
for $\beta \in {\mathcal A}_1$, and
$$ R'_{\beta'} \coloneqq  R_\beta(\vec \m + \vec \h'', \vec \h)$$
for $\beta' \in {\mathcal A}'_1$ the copy of an element $\beta \in {\mathcal A}_1$.
\item The distinguished node stays at $\alpha_0$.
\item The inactive nodes consist of all the nodes in ${\mathcal A}_0$, together with all the previously inactive nodes $\beta$ of ${\mathcal A}_1$, as well as their copies $\beta'$ in ${\mathcal A}'_1$.
\item The realisations $f'_\beta$ for $\beta \in {\mathcal A}' = {\mathcal A}_0 \cup {\mathcal A}_1 \cup {\mathcal A}'_1$ are defined by setting
$$ f'_\beta \coloneqq  \nu_{b_\beta} + 1$$
for $\beta \in {\mathcal A}_0$, and
$$ f'_\beta = f'_{\beta'} = f_\beta$$
for $\beta \in {\mathcal A}_1$, where $\beta'$ is the copy of $\beta$ in ${\mathcal A}'_1$.
\end{itemize}

It is a routine matter to check that ${\mathcal S}'$ obeys the axioms required for a polynomial system; its number of nodes and fine degrees of freedom are bounded by a constant $C'_0$ depending only on $C_0, r$, and the polynomials have coefficients bounded by $A^{C'_1}$ for some $C'_1$ depending only on $C_1$.  Similarly one verifies that $\vec f'$ is indeed a realization of ${\mathcal S}'$.

Let $d_*$ be the minimal distance of an active node of ${\mathcal A}_1 \cup {\mathcal A}'_1$ to $\alpha$ (or equivalently, the minimal $\m$-degree of $R'_\beta$ as $\beta$ ranges over active nodes in ${\mathcal A}_1 \cup {\mathcal A}'_1$); this is well-defined since $\alpha_0 \in {\mathcal A}_1$ is active.  Amongst all the active nodes in ${\mathcal A}_1 \cup {\mathcal A}'_1$ at distance $d_*$ from $\alpha$, let $\tilde \alpha$ be a node that maximizes its distance from $\alpha_0$.  From the ultratriangle inequality (and the fact that $\alpha_0$ is already at a maximal distance from $\alpha$ amongst all active nodes) we see that any active node in ${\mathcal A}_1 \cup {\mathcal A}'_1$ that is further than distance $d_*$ from $\alpha$ will be no further from $\alpha_0$ than $\tilde \alpha$.  Thus $\tilde \alpha$ is an extreme node in ${\mathcal S}'$.  

Observe that if $\beta$ is a node in ${\mathcal A}_1$ and $\beta'$ its copy in ${\mathcal A}'_1$, then $R'_\beta - R'_{\beta'}$ has $\m$-degree strictly less than that of $R'_\beta$, thus the distance between $\beta$ and $\beta'$ is less than that between $\beta$ and $\alpha$.  From this and the ultratriangle inequality we see that for $i > d_*$, the number of equivalence classes in ${\mathcal S}'$ relative to $\tilde \alpha$ at distance $i$ is equal to the number of equivalence classes in ${\mathcal S}$ relative to $\alpha$, while for $i = d_*$, ${\mathcal S}'$ has one fewer equivalence class relative to $\tilde \alpha$ at distance $d_*$ than ${\mathcal S}$ relative to $\alpha$.  Thus the weight vector $\vec w'$ of ${\mathcal S}'$ relative to $\tilde \alpha$ is less than the weight vector $\vec w$ of ${\mathcal S}$ relative to $\alpha$.  Furthermore, given the bounds on the number of nodes and fine degrees of freedom, the weight vector $\vec w'$ ranges in a finite set that depends on $d,r,C_0,C_1,\vec w$.  Applying the hypothesis $P( C_0, C_1, \vec w')$, we obtain the claim.

\subsection{Parallelopipedization}

It remains to establish the linear case of Theorem \ref{l3}.  As in \cite{tz-pattern}, the next step is ``parallelopipedization'', in which one repeatedly uses the Cauchy-Schwarz inequality to reduce matters to controlling a weighted averaged local Gowers norm.

Let the notation be as in Theorem \ref{l3}, with ${\mathcal S}$ linear.  We abbreviate $(\h_1,\dots,\h_{D_{\mathcal S}})$ as $\vec h$.  By subtracting $R_{\alpha_0}$ from all of the other polynomials $R_\alpha$, we may assume that $R_{\alpha_0}=0$, so that all the active $R_\alpha$ have degree at most one in $\m$.  We write ${\mathcal A}_l$ for those nodes at distance one from $\alpha_0$, and ${\mathcal A}_{nl}$ for all nodes at distance greater than one, thus ${\mathcal A}$ is partitioned into $\{\alpha_0\}$, ${\mathcal A}_l$, ${\mathcal A}_{nl}$, with ${\mathcal A}_{nl}$ consisting entirely of inactive nodes, so that for each $\alpha \in {\mathcal A}_{nl}$, $f_\alpha$ is equal either to $1$ or to $\nu_{b_\alpha}$ for some $b_\alpha \in [W]$ coprime to $W$.  By deleting all nodes with $f_\alpha=1$ we can assume that only the latter case $f_\alpha = \nu_{b_\alpha}$ occurs for $\alpha \in {\mathcal A}_{nl}$.

For each $\alpha \in {\mathcal A}_l$ one has
$$ R_\alpha = \vec a_\alpha \cdot \vec \m + c_\alpha$$
for some non-zero $\vec a_\alpha \in \Z[\vec h]^r$ and some $c_\alpha \in \Z[\vec h]$, with $\cdot$ denoting the usual dot product; furthermore from the axioms of a linear system we see that the $\vec a_\alpha$ are distinct as $\alpha$ varies.  We can then write $\Lambda_{\vec S}(\vec f)$ as
$$\E_{\vec h \in [M]^{D_{\mathcal S}}} \E_{\vec m \in [M_0]^r} \E_{x \in \Z/N\Z} f_{\alpha_0}(x) \left(\prod_{\alpha \in {\mathcal A}_{nl}} \nu_{b_\alpha}(x + R_\alpha(\vec m, \vec h))\right) \left(\prod_{\alpha \in {\mathcal A}_l} f_\alpha( x + \vec a_\alpha(\vec h) \cdot \vec \m + c_\alpha(\vec h) \right).$$
We need to show that this expression is $o(1)$.  Arguing as in the previous section, we may replace the set $[M_0]^r$ that $\vec m$ is being averaged over by the multiset
$$ [M_0]^r + \sum_{\alpha \in {\mathcal A}_l} [M]^r $$
(that is to say, the sum of $[M_0]^r$ and $|{\mathcal A}_l|$ copies of $[M]^r$, counting multiplicity.  Thus it suffices to show that the expression
\begin{equation}\label{exd}
\begin{split}
&\E_{\vec h \in [M]^{D_{\mathcal S}}} \E_{\vec m \in [M_0]^r} \E_{\vec k_\alpha \in [M]^r \forall \alpha \in {\mathcal A}_l} \E_{x \in \Z/N\Z} f_{\alpha_0}(x) \left(\prod_{\alpha \in {\mathcal A}_{nl}} \nu_{b_\alpha}\left(x + R_\alpha(\vec m + \sum_{\beta \in {\mathcal A}_l} \vec k_\beta, \vec h\right)\right)\\
&\quad \left (\prod_{\alpha \in {\mathcal A}_l} f_\alpha\left( x + \vec a_\alpha(\vec h) \cdot \vec \m + \sum_{\beta \in {\mathcal A}_l} \vec a_\alpha(\vec h) \cdot \vec k_\beta + c_\alpha(\vec h) \right) \right)
\end{split}
\end{equation}
is $O( \|f_{\alpha_0}\|^c + o(1))$.  We shift $x$ by $-\sum_{\beta \in {\mathcal A}_l} \vec a_\beta(\vec h) \cdot \vec k_\beta$ to write this expression as
$$\E_{\vec h \in [M]^{D_{\mathcal S}}} \E_{\vec m \in [M_0]^r} \E_{\vec k_\alpha \in [M]^r \forall \alpha \in {\mathcal A}_l} \E_{x \in \Z/N\Z} f_{\alpha_0,\vec m, \vec h, \vec k}(x) \prod_{\alpha \in {\mathcal A}_l} f_{\alpha,\vec m, \vec h, \vec k}(x) $$
where $\vec k \coloneqq  (\vec k_\alpha)_{\alpha \in {\mathcal A}_l}$, with
$$  f_{\alpha_0,\vec m, \vec h, \vec k}(x) \coloneqq  f_{\alpha_0}\left(x - \sum_{\beta \in {\mathcal A}_l} \vec a_\beta(\vec h) \cdot \vec k_\beta\right)
\prod_{\alpha \in {\mathcal A}_{nl}} \nu_{b_\alpha}\left(x + R_\alpha\left(\vec m + \sum_{\beta \in {\mathcal A}_l} \vec m_\beta, \vec h\right)-\sum_{\beta \in {\mathcal A}_l} \vec a_\beta(\vec h) \cdot \vec k_\beta\right)$$
and
$$ f_{\alpha,\vec m, \vec h, \vec k}(x) \coloneqq  f_\alpha\left( x + \vec a_\alpha(\vec h) \cdot \vec \m + \sum_{\beta \in {\mathcal A}_l} (\vec a_\alpha(\vec h) - \vec a_\beta(\vec h)) \cdot \vec k_\beta + c_\alpha(\vec h) \right)
$$
for $\alpha \in {\mathcal A}_l$.  The key point here is that $f_{\alpha,\vec m, \vec h, \vec k}$ does not depend on the $\alpha$ component $\vec k_\alpha$ of $\vec k$.  We also have the pointwise bound
$$ |f_{\alpha,\vec m, \vec h, \vec k}(x)| \leq \nu_{\alpha,\vec m, \vec h, \vec k}(x)$$
for all $\alpha \in {\mathcal A}_l$, where $\nu_{\alpha,\vec m, \vec h, \vec k}(x)$ is either identically equal to $1$, or is given by the formula
$$ \nu_{\alpha,\vec m, \vec h, \vec k}(x) = \nu_{b_\alpha}\left( x + \vec a_\alpha(\vec h) \cdot \vec \m + \sum_{\beta \in {\mathcal A}_l} (\vec a_\alpha(\vec h) - \vec a_\beta(\vec h)) \cdot \vec k_\beta + c_\alpha(\vec h) \right).
$$
For sake of exposition we will assume that the latter holds for all $\alpha$, as this is the most difficult case.  As with $f_{\alpha,\vec m, \vec h, \vec k}$, the quantity $\nu_{\alpha,\vec m,\vec h, \vec k}(x)$ does not depend on the $\alpha$ component of $\vec k$.

Applying the weighted Cauchy-Schwarz-Gowers inequality (see \cite[Proposition A.2]{tz-pattern} or \cite[Corollary B.4]{gt-linear}), we can thus bound upper bound the absolute value of \eqref{exd} by
$$ \E_{\vec h \in [M]^{D_{\mathcal S}}} \E_{\vec m \in [M_0]^r} \E_{x \in \Z/N\Z} \| f_{\alpha_0,\vec m, \vec h, \cdot}(x) \|_{\Box^{{\mathcal A}_l}(\nu)} \prod_{\alpha \in {\mathcal A}_l} \| \nu_{\alpha,\vec m, \vec h, \cdot}(x) \|_{\Box^{{\mathcal A}_l \backslash \alpha}}^{1/2} $$
where
\begin{align*}
 \| f_{\alpha_0,\vec m, \vec h, \cdot}(x) \|_{\Box^{{\mathcal A}_l}(\nu)}^{2^{|{\mathcal A}_l|}}
&\coloneqq  \E_{\vec k^{(0)}, \vec k^{(1)} \in ([M]^r)^{{\mathcal A}_l}} \left[ \prod_{\omega \in \{0,1\}^{{\mathcal A}_l}} f_{\alpha_0,\vec m, \vec h, \vec k^{(\omega)}}(x) \right] \\
&\quad \times \prod_{\alpha \in {\mathcal A}_l} \prod_{\omega \in \{0,1\}^{{\mathcal A}_l \backslash \{\alpha\}}} \nu_{\alpha, \vec m, \vec h, \vec k^{(\omega)}}(x)
\end{align*}
and
\begin{align*}
\| \nu_{\alpha,\vec m, \vec h, \cdot}(x) \|_{\Box^{{\mathcal A}_l \backslash \alpha}}^{2^{|{\mathcal A}_l|-1}}
&\coloneqq  \E_{\vec k^{(0)}, \vec k^{(1)} \in ([M]^r)^{{\mathcal A}_l \backslash \{\alpha\}}}
\prod_{\omega \in \{0,1\}^{{\mathcal A}_l \backslash \{\alpha\}}}  \nu_{\alpha,\vec m, \vec h, \vec k^{(\omega)}}(x), 
\end{align*}
where for $\vec k^{(0)} = (\vec k^{(0)}_\beta)_{\beta \in {\mathcal A}_l}$, $\vec k^{(1)} = (\vec k^{(1)}_\beta)_{\beta \in {\mathcal A}_l}$, and $\omega = (\omega_\beta)_{\beta \in {\mathcal A}_l}$ one has
$$ \vec k^{(\omega)} \coloneqq  (\vec k^{(\omega_\beta)}_\beta)_{\beta \in {\mathcal A}_l};$$
if $\vec k^{(0)}, \vec k^{(1)}$ is only given in $([M]^r)^{{\mathcal A}_l \backslash \{\alpha\}}$, we extend it arbitrarily to $([M]^r)^{{\mathcal A}_l}$ for the purposes of defining $\vec k^{(\omega)}$, and similarly if $\omega$ is only given in $\{0,1\}^{{\mathcal A}_l \backslash \{\alpha\}}$ instead of $\{0,1\}^{{\mathcal A}_l}$.  This leaves the $\alpha$ component of $\vec k^{(\omega)}$ undefined, but this is irrelevant for the purposes of evaluating $\nu_{\alpha, \vec m, \vec h, \vec k^{(\omega)}}(x)$ because (as mentioned previously) this quantity does not depend on the $\alpha$ component of $\vec k^{(\omega)}$.

In Appendix \ref{polyapp} we will establish the following estimate:

\begin{lemma}\label{gopw}  With the notation as above, we have
\begin{equation}\label{gopw-eq}
\E_{\vec h \in [M]^{D_{\mathcal S}}} \E_{\vec m \in [M_0]^r} \E_{x \in \Z/N\Z} \| \nu_{\alpha,\vec m, \vec h, \cdot}(x) \|_{\Box^{{\mathcal A}_l \backslash \alpha}}^{2^{|{\mathcal A}_l|-1}} = 1 + o(1).
\end{equation}
\end{lemma}

Thus, by H\"older's inequality (and modifying $c$ as necessary), to show that the expression in \eqref{exd} is $o(1)$, it suffices to establish the bound
\begin{equation}\label{scl}
 \E_{\vec h \in [M]^{D_{\mathcal S}}} \E_{\vec m \in [M_0]^r} \E_{x \in \Z/N\Z} \| f_{\alpha_0,\vec m, \vec h, \cdot}(x) \|_{\Box^{{\mathcal A}_l}(\nu)}^{2^{|{\mathcal A}_l|}} \ll \| f_{\alpha_0} \|^c + o(1).
\end{equation}

This is a weighted version of Theorem \ref{l1}, and will be deduced from that theorem by one final application of \eqref{scl}, to which we now turn.

\subsection{Final Cauchy-Schwarz}

By definition, the left-hand side of \eqref{scl} expands as
$$
 \E_{\vec h \in [M]^{D_{\mathcal S}}} \E_{x \in \Z/N\Z} \E_{\vec k^{(0)}, \vec k^{(1)} \in ([M]^r)^{{\mathcal A}_l}} w(\vec h, \vec k^{(0)}, \vec k^{(1)}, x) \prod_{\omega \in \{0,1\}^{{\mathcal A}_l}} 
f_{\alpha_0}(x - \sum_{\beta \in {\mathcal A}_l} \vec a_\beta(\vec h) \cdot \vec k^{(\omega_\beta)}_\beta) $$
where\footnote{In the analogous expansion in \cite[\S 5.19]{tz-pattern}, the terms arising from $\alpha \in {\mathcal A}_l$ were mistakenly omitted.}
\begin{align*}
 &w(\vec h, \vec k^{(0)}, \vec k^{(1)}, x) \coloneqq 
 \E_{\vec m \in [M_0]^r}  
\left(\prod_{\alpha \in {\mathcal A}_{nl}} \prod_{\omega \in \{0,1\}^{{\mathcal A}_l}}  \nu_{b_\alpha}(x + R_\alpha(\vec m + \sum_{\beta \in {\mathcal A}_l} \vec k^{(\omega_\beta)}_\beta, \vec h))-\sum_{\beta \in {\mathcal A}_l} \vec a_\beta(\vec h) \cdot \vec k_\beta) \right)\\
& \times \left(\prod_{\alpha \in {\mathcal A}_l} \prod_{\omega \in \{0,1\}^{{\mathcal A}_l \backslash \{\alpha\}}} \nu_{\alpha, \vec m, \vec h, \vec k^{(\omega)}}(x)\right).
\end{align*}
On the other hand, if we identify ${\mathcal A}_l$ with $\{1,\dots,D\}$, then $D \leq C_0$ and the expression
$$
 \E_{\vec h \in [M]^{D_{\mathcal S}}} \E_{x \in \Z/N\Z} \E_{\vec k^{(0)}, \vec k^{(1)} \in ([M]^r)^{{\mathcal A}_l}} w(\vec h, \vec k^{(0)}, \vec k^{(1)}, x) \prod_{\omega \in \{0,1\}^{{\mathcal A}_l}} 
f_{\alpha_0}\left(x - \sum_{\beta \in {\mathcal A}_l} \vec a_\beta(\vec h) \cdot \vec k^{(\omega_\beta)}_\beta\right) $$
can be bounded in magnitude using the Cauchy-Schwarz-Gowers inequality \eqref{csg} by
$$  \E_{\vec h \in [M]^{D_{\mathcal S}}} \| f_{\alpha_0} \|_{\Box^{rD}_{(a_{ij}(\vec h) [M])_{1 \leq i \leq D; 1 \leq j \leq r}}}^{2^{rD}} $$
where $a_{i1},\dots,a_{ir} \in \Z[\vec \h]$ are the components of $a_i$.  By H\"older's inequality, this is bounded by
$$  (\E_{\vec h \in [M]^{D_{\mathcal S}}} \| f_{\alpha_0} \|_{\Box^{rD}_{(a_{ij}(\vec h) [M])_{1 \leq i \leq D; 1 \leq j \leq r}}})^{2^{rD}} $$
which is an expression of the form $\|f_{\alpha} \|^{2^{rD}}$.  Thus, by the triangle inequality, it suffices to show that
$$
 \E_{\vec h \in [M]^{D_{\mathcal S}}} \E_{x \in \Z/N\Z} \E_{\vec k^{(0)}, \vec k^{(1)} \in ([M]^r)^{{\mathcal A}_l}} (w(\vec h, \vec k^{(0)}, \vec k^{(1)}, x)-1) \prod_{\omega \in \{0,1\}^{{\mathcal A}_l}} 
f_{\alpha_0}\left(x - \sum_{\beta \in {\mathcal A}_l} \vec a_\beta(\vec h) \cdot \vec k^{(\omega_\beta)}_\beta\right)= o(1).$$
 Bounding $f_{\alpha_0}$ in magnitude by $\nu_{b_{\alpha_0}}+1$ and using Cauchy-Schwarz, it suffices to show that
\begin{align*}
& \E_{\vec h \in [M]^{D_{\mathcal S}}} \E_{x \in \Z/N\Z} \E_{\vec k^{(0)}, \vec k^{(1)} \in ([M]^r)^{{\mathcal A}_l}} (w(\vec h, \vec k^{(0)}, \vec k^{(1)}, x)-1)^2 \prod_{\omega \in \{0,1\}^{{\mathcal A}_l}} \\
&\left(\nu_{b_{\alpha_0}}(x - \sum_{\beta \in {\mathcal A}_l} \vec a_\beta(\vec h) \cdot \vec k^{(\omega_\beta)}_\beta)+1\right)= o(1).
\end{align*}
Expanding out the square, it suffices to show that
\begin{align*}
& \E_{\vec h \in [M]^{D_{\mathcal S}}} \E_{x \in \Z/N\Z} \E_{\vec k^{(0)}, \vec k^{(1)} \in ([M]^r)^{{\mathcal A}_l}} w(\vec h, \vec k^{(0)}, \vec k^{(1)}, x)^j \prod_{\omega \in \{0,1\}^{{\mathcal A}_l}} \\
&\left(\nu_{b_{\alpha_0}}(x - \sum_{\beta \in {\mathcal A}_l} \vec a_\beta(\vec h) \cdot \vec k^{(\omega_\beta)}_\beta)+1\right)= 2^{|{\mathcal A}_l|} + o(1)
\end{align*}
for $j=0,1,2$.

We just treat the most difficult case $j=2$, as the other cases $j=0,1$ are similar (with fewer $\nu$-type factors).  Expanding out the second product, it suffices to show that
$$
 \E_{\vec h \in [M]^{D_{\mathcal S}}} \E_{x \in \Z/N\Z} \E_{\vec k^{(0)}, \vec k^{(1)} \in ([M]^r)^{{\mathcal A}_l}} w(\vec h, \vec k^{(0)}, \vec k^{(1)}, x)^2 \prod_{\omega \in \Omega} 
\nu_{b_{\alpha_0}}(x - \sum_{\beta \in {\mathcal A}_l} \vec a_\beta(\vec h) \cdot \vec k^{(\omega_\beta)}_\beta)= 1 + o(1)$$
for any $\Omega \subset \{0,1\}^{{\mathcal A}_l}$.  The left-hand side may be expanded as
\begin{equation}\label{task}
\begin{split}
 &\E_{\vec h \in [M]^{D_{\mathcal S}}} \E_{x \in \Z/N\Z} \E_{\vec k^{(0)}, \vec k^{(1)} \in ([M]^r)^{{\mathcal A}_l}} 
 \E_{\vec m^{(0)}, \vec m^{(1)} \in [M_0]^r}  \prod_{i=0}^1 \\
&\quad \left(\prod_{\alpha \in {\mathcal A}_{nl}} \prod_{\omega \in \{0,1\}^{{\mathcal A}_l}} \nu_{b_\alpha}(x + R_\alpha(\vec m^{(i)} + \sum_{\beta \in {\mathcal A}_l} \vec m^{(i)}_\beta, \vec h)-\sum_{\beta \in {\mathcal A}_l} \vec a_\beta(\vec h) \cdot \vec k_\beta) \right)\\
& \quad \times \prod_{\alpha \in {\mathcal A}_l} \prod_{\omega \in \{0,1\}^{{\mathcal A}_l \backslash \{\alpha\}}} 
\nu_{b_\alpha}( x + \vec a_\alpha(\vec h) \cdot \vec m^{(i)} + \sum_{\beta \in {\mathcal A}_l} (\vec a_\alpha(\vec h) - \vec a_\beta(\vec h)) \cdot \vec k^{(\omega_\beta)}_\beta + c_\alpha(\vec h) ).
\end{split}
\end{equation}
But this is $1+o(1)$ thanks to the polynomial forms property of the measures $\nu_b$ (see Appendix \ref{polyapp}).  This completes the proof of Theorem \ref{l3} and hence Theorem \ref{l2}.

\section{The $W$-trick}\label{w-trick}

Theorem \ref{l2} has a particularly pleasant consequence in the setting where all polynomials involved are distinct to top order.

\begin{corollary}\label{l2-cor}  Let $d,r,D, W, A, N$ be as in previous sections.  For each $i=1,\dots,D$, let $b_i \in [W]$ be coprime to $W$.   For each $i=1,\dots,d$, let $R_i \in \Z[\m_1,\dots,\m_r]$ be a polynomial of degree at most $d$, with all coefficients bounded in magnitude by $A$.  Assume also that for $1 \leq i < j \leq D$, $R_i - R_j$ has degree exactly $d$.  Let $M_0$ be a quantity with $A^{-1} N^{1/d} \leq M_0 \leq A N^{1/d}$.  Then
$$ \E_{x \in \Z/N\Z} \E_{\vec m \in [M_0]^r} \prod_{i=1}^D \Lambda'_{b_i,W}( x + R_i( \vec m ) ) = 1 + o(1).$$
\end{corollary}

\begin{proof}  The claim is equivalent to
$$ \E_{x \in \Z/N\Z} \E_{\vec m \in [M_0]^r} \left(\prod_{i=1}^D \Lambda'_{b_i,W}( x + R_i( \vec m ) )\right) - 1 = o(1).$$
By writing $\Lambda'_{b_i,W} = 1 + (\Lambda'_{b_i,W}-1)$, one can expand the left-hand side as the sum of $2^D-1$ terms, each of which is bounded in magnitude by $O( \| \Lambda'_{b_i,W}-1\|^c ) + o(1)$ for some $i$ and some fixed $c>0$ by Theorem \ref{l2} (after permuting the indices), where the norm is of the form \eqref{sof} with $d_0=d$.  On the other hand, from Theorem \ref{l1} we have $\| \Lambda'_{b_i,W}-1\| = o(1)$.  The claim follows.
\end{proof}

With this corollary and the ``$W$-trick'' (as in \cite[\S 5]{gt-linear}), we can now prove Theorem \ref{main}.  Let $d,r,k,P_1,\dots,P_r, M()$ be as in Theorem \ref{main}.  We let $d,r,D,W,A,R,N,N'$ be as in previous sections, with $d,r$ as previously chosen and $D$ set equal to $k$.  By replacing $N$ by $N'$, and using the convergence of the infinite product $\prod_p \beta_p$ and the fact that $w$ goes to infinity, it suffices to show that
$$
\E_{n' \in [N']} \E_{m' \in [M']^r} \Lambda(n'+P_1(m')) \dots \Lambda(n'+P_k(m')) = \prod_{p \leq w} \beta_p + o(1)$$
where $M' \coloneqq  M(N')$.  

The contribution to the left-hand side of the case when one of the $n' + P_i(m')$ is a prime power (rather than a prime) can easily be seen to be $o(1)$ (in fact one obtains a power savings in $N'$), so we may replace $\Lambda$ by its restriction $\Lambda'$ to the primes without loss of generality.

We split $n'$ and $m'$ into residue classes $n' = b\ (W)$ and $m' = c\ (W)$ for $b \in [W]$ and $c \in [W]^r$.  Call a pair $(b,c)$ \emph{admissible} if $b + P_i(c)$ is coprime to $W$ for all $i=1,\dots,k$.  From the Chinese remainder theorem, we see that the number of admissible pairs is $W^{r+1} \left(\frac{\phi(W)}{W}\right)^k \prod_{p \leq w} \beta_p$.  For inadmissible $(b,c)$, the quantity $\Lambda'(n'+P_1(m')) \dots \Lambda'(n'+P_k(m'))$ is only non-vanishing when one of the $n'+P_i(m')$ is a prime $p \leq w$.  It is not difficult to see that the total contribution of such a case is $o(1)$ if $w$ is sufficiently slowly growing with respect to $N'$.  Thus we may restrict attention to admissible pairs $(b,c)$.  Approximating the arithmetic progression $\{ n' \in [N']: n' = b\ (W)\}$ by $\{ Wn+b: n \in [N]\}$, and similarly approximating $\{ m' \in [M']^r: m' = c\ (W)\}$ by $\{ Wm + c: m \in [M]^r\}$ with $M \coloneqq  \lfloor M'/W \rfloor$ (using crude estimates to bound the error in this approximation by multiplicative and additive errors of $o(1)$, assuming $w$ sufficiently slowly growing), it will thus suffice to show that
\begin{equation}\label{sad}
\E_{n \in [N]} \E_{m \in [M]^r} \Lambda'(Wn+b+P_1(Wm+c)) \dots \Lambda'(Wn+b+P_k(Wm+c)) = \left(\frac{W}{\phi(W)}\right)^k (1 + o(1))
\end{equation}
uniformly for all admissible pairs $(b,c)$.

As $(b,c)$ is admissible, we can write
$$ Wn + b + P_i(Wm+c) = W(n + R_i(m)) + b_i$$
for some $b_i \in [W]$ coprime to $W$, and some polynomial $R_i \in \Z[\m_1,\dots,\m_r]$ of degree $d$ with all coefficients bounded in magnitude by $A$.  Since the $P_i - P_j$ all had degree $d$, the $R_i - R_j$ do also.  Recalling that $M = o(N^{1/d})$, we see that the quantities $n + R_i(m)$ will lie in the interval $[R,N]$ unless $n = o(N)$ or $n = N - o(N)$.  The contribution of these latter cases to \eqref{sad} can easily be verified to be $o( (\frac{W}{\phi(W)})^k )$ by any standard upper bound sieve (e.g. the Selberg sieve, or the ``fundamental lemma of sieve theory'', see e.g. \cite[Theorem 6.12]{fi}).  Using \eqref{bawn}, we thus see that \eqref{sad} is equivalent to the estimate
$$
\E_{n \in \Z/N\Z} \E_{m \in [M]^r} \Lambda'_{b_1,W}(n+R_1(m)) \dots \Lambda'_{b_k,W}(n+R_k(m)) = 1 + o(1).$$
But this follows from Corollary \ref{l2-cor} (noting that the lower bound on $M'$ will imply that $M \geq A^{-1} N^{1/d}$ if $\omega$ is going to zero sufficiently slowly).  This concludes the proof of Theorem \ref{main}.

We now adapt the above arguments to prove Theorem \ref{nn2}.  Repeating the above arguments all the way up to \eqref{sad}, we arrive at the task of showing that
$$
\E_{n \in [N]} \E_{m \in [M]} \Lambda'(Wn+b) \Lambda'(Wn+b+(Wm+c)) \Lambda'(Wn+b+P_3(Wm+c)) = \left(\frac{W}{\phi(W)}\right)^3 (1 + o(1))
$$
uniformly for $(b,c) \in [W]^2$ with $b, b+c, b+P_3(c)$ coprime to $W$, with $M = o(\sqrt{N})$ and $M \geq \omega(N) N$ for some function $\omega(N)$ that goes to zero sufficiently slowly.  Continuing the above arguments, we then reduce to showing that
$$
\E_{n \in \Z/N\Z} \E_{m \in [M]} \Lambda'_{b_1,W}(n) \Lambda'_{b_2,W}(n + R_2(m)) \Lambda'_{b_3,W}(n + R_3(m)) = 1 + o(1) $$
where $b_1,b_2,b_3 \in [W]$ are given by the congruences
\begin{align*}
b_1 &= b\ (W) \\
b_2 &= b+c\ (W) \\
b_3 &= b+P_3(c)\ (W)
\end{align*}
and $R_2,R_3$ are the polynomials
\begin{align*}
R_2 &\coloneqq  \m + \frac{b+c-b_2}{W} \\
R_3 &\coloneqq  \frac{P_3(W\m+c) - b_3}{W}.
\end{align*}
Using Theorem \ref{l2}, we see that
$$
\E_{n \in \Z/N\Z} \E_{m \in [M]} \Lambda'_{b_1,W}(n) \Lambda'_{b_2,W}(n + R_2(m)) (\Lambda'_{b_3,W}(n + R_3(m))-1) = o(1) $$
so it suffices to show that
$$
\E_{n \in \Z/N\Z} \E_{m \in [M]} \Lambda'_{b_1,W}(n) \Lambda'_{b_2,W}(n + R_2(m)) = 1+o(1) $$
which on reversing some of the arguments following \eqref{sad} is equivalent to
$$ 
\E_{n \in [N]} \E_{m \in [M]} \Lambda'(Wn+b) \Lambda'(Wn+b+(Wm+c)) = \left(\frac{W}{\phi(W)}\right)^2 (1+o(1)).$$
Set $M_0 \coloneqq  \lfloor M \log^{-10} N \rfloor$.  Using crude bounds on $\Lambda'$, we may replace the average $[M]$ by $[M]-[M_0]$ with negligible error, and then by shifting $n$ by an element of $M_0$ and incurring a further negligible error, we may reduce to showing that
$$ 
\E_{n \in [N]} \E_{h \in [M_0]} \E_{m \in [M]} \Lambda'(W(n+h)+b) \Lambda'(Wn+b+(Wm+c)) = \left(\frac{W}{\phi(W)}\right)^2 (1+o(1)).$$
The left-hand side factors as
$$ 
\E_{n \in [N]} (\E_{h \in [M_0]} \Lambda'(W(n+h)+b)) (\E_{m \in [M]} \Lambda'(Wn+b+(Wm+c))).$$
From the prime number theorem in arithmetic progressions we have
$$ 
\E_{n \in [N]} \E_{m \in [M]} \Lambda'(Wn+b+(Wm+c)) = \left(\frac{W}{\phi(W)}\right) (1+o(1))$$
if $w$ is sufficiently slowly growing, so it suffices to show that
$$ \E_{n \in [N]} (\E_{h \in [M_0]} \frac{\phi(W)}{W} \Lambda'(W(n+h)+b)-1) (\frac{\phi(W)}{W} \E_{m \in [M]} \Lambda'(Wn+b+(Wm+c))) = o(1).$$
From the Brun-Titchmarsh inequality, the expression $\frac{\phi(W)}{W} \E_{m \in [M]} \Lambda'(Wn+b+(Wm+c))$ is bounded by $O(1)$, so by the Cauchy-Schwarz inequality it suffices to show that
$$ \E_{n \in [N]} \left(\E_{h \in [M_0]} \frac{\phi(W)}{W} \Lambda'(W(n+h)+b)-1\right)^2 = o(1).$$
Since $d \leq 5$, we have $M_0 \geq N^{1/6+\eps}$, and the above claim then follows from standard zero density estimates (see e.g. \cite[Theorem 1.1]{kou}).  On the generalized Riemann hypothesis\footnote{In fact it suffices to assume the generalized density hypothesis.}, one can obtain this claim for $M_0$ as low as $N^\eps$, and then no restriction on $d$ is necessary; again, see \cite[Theorem 1.1]{kou}.

\begin{remark}  The same method lets us handle a triplet of polynomials $P_1,P_2,P_3 \in \Z[\m]$ in which $P_3-P_1$ has degree $d$ and $P_2-P_1$ has degree $k$ for some $1 \leq k < d$ with $k/d > 1/6$ (and the hypothesis $k/d>1/6$ can be omitted on the generalized Riemann or density hypothesis).  Indeed, the above arguments let us reduce to showing an estimate of the form 
$$ 
\E_{n \in [N]} \E_{m \in [M]} \Lambda'(Wn+b) \Lambda'(Wn+b+(P_2-P_1)(Wm+c)) = \left(\frac{W}{\phi(W)}\right)^2 (1+o(1)),$$
and a standard application of the circle method (using some Fourier restriction theorem for the von Mangoldt function on short intervals) lets us control this expression in turn by averages of $\Lambda$ on arithmetic progressions of spacing $O(W^{O(1)})$ and length roughly $N^{k/d}$, which can again be controlled by zero density estimates as before.  We leave the details to the interested reader.

One may in principle be able to handle some higher complexity patterns of this type, e.g. $P_1 = 0, P_2 = \m^k, P_3 = 2\m^k, P_4 = \m^d$ when $1 \leq k < d$ with $k/d$ sufficiently close to $1$.  Morally speaking, after using some suitable adaptation of the arguments in this paper and the (non-trivial) fact that the pattern $n, n+m^k, n+2m^k$ has ``true complexity'' $1$ in the sense of Gowers and Wolf \cite{gow-wolf}, the average corresponding to this set of polynomials should be controlled by an expression roughly of the form
$$ \E_{n \in [N]} \sum_{\alpha \in \R/\Z} |\E_{m \in [M^k]} (\Lambda'(W(n+m)+b)-1) e(\alpha m)| $$
and one would expect to be able to control this quantity when $k/d$ is large from existing analytic number theory methods; the best result in this direction we currently know of is by Zhan \cite{zhan}, who used moment bounds on $L$-functions to treat the case $k/d > 5/8$.    Again, we will not pursue the details of these arguments further here.
\end{remark}

\appendix

\section{The polynomial forms condition}\label{polyapp}

Throughout this appendix, the parameters $d,r,D,\kappa,W,A,R,N$ are as in Section \ref{agu}, with all quantities below allowed to depend on $d,r$.  We will now prove the various polynomial forms conditions required on the functions $\nu_b$, specifically Proposition \ref{gu}, Lemma \ref{ps2}, Lemma \ref{st-lemma}, Lemma \ref{gopw}, and showing that the quantity \eqref{task} is $1+o(1)$.

Our starting point is the following estimate on the $\nu_b$ from \cite{tz-pattern}, \cite{tz-narrow}:

\begin{theorem}[Polynomial forms condition]\label{poly}  Let $k,s,C$ be natural numbers not depending on $N$, let $b_1,\dots,b_s \in [W]$ be coprime to $W$, and let $P_1,\dots,P_k \in \Z[\m_1,\dots,\m_s]$ be distinct polynomials of degree at most $d$, with all non-constant coefficients of size at most $A^C$ in magnitude.  Assume that $\kappa$ is sufficiently small depending on $k,s$.  Assume also that the $P_i-P_j$ are non-constant for all $1 \leq i < j \leq s$.  Let $M_1,\dots,M_s$ be quantities with $\log^{1/\kappa} N \leq M_1,\dots,M_s \leq N$.  Then
$$ \E_{x \in \Z/N\Z} \E_{m_1 \in [M_1],\dots,m_s \in [M_s]} \prod_{i=1}^k \nu_{b_i}(x + P_i(m_1,\dots,m_s)) = 1+o(1).$$
\end{theorem}

\begin{proof}  In the case $b_1=\dots=b_s$ (and when the constant coefficients of the $P_i$ are also bounded in magnitude by $A^C$), this follows directly from \cite[Corollary 11.2]{tz-pattern} in the case when the $M_1,\dots,M_s$ are bounded below by (say) $N^\kappa$, and from \cite[Proposition 3]{tz-narrow} in the general case\footnote{Strictly speaking, \cite[Proposition 3]{tz-narrow} only claims the case when $M_1=\dots=M_s = \log^L N$ for some sufficiently large $L$ depending on $k,s$, but the arguments easily extend to larger values of $M_1,\dots,M_S$.}.  The arguments in \cite[\S 10, 11]{tz-pattern} used to prove \cite[Corollary 11.2]{tz-pattern} or \cite[Proposition 3]{tz-narrow} can be easily modified to handle the case when the more general case when the $b_i$ are permitted to be distinct and the constant coefficients are permitted to be large, after replacing every occurrence of $W P_j + b$ with $W P_j + b_j$ in these arguments.  (The notion of a ``terrible'' prime has to then be modified to be a prime $p>w$ that divides $(WP_i + b_i) - (WP_j + b_j)$ for some $1 \leq i < j \leq s$, rather than just dividing $P_i - P_j$; however, these polynomials $(WP_i + b_i) - (WP_j + b_j)$ are non-constant with all non-constant coefficients $O( W A^C )$, and the arguments used to prove \cite[Corollary 11.2]{tz-pattern} or \cite[Proposition 3]{tz-narrow} still show that the total contribution of the terrible primes only contributes a multiplicative factor of $O(1)$ to the error.
\end{proof}

Lemma \ref{ps2} then follows from this theorem (in the $b_1=\dots=b_s=b$ case), as after expanding out the fourth power one obtains a sum of $4^{2^D-1}$ expressions, all of which are $1+o(1)$ thanks to Theorem \ref{poly} (with $s=4(r+D)$, and $k$ at most $4 (2^D-1)$).  Proposition \ref{gu} similarly follows from this theorem (in the $b_1=\dots=b_s=b$ case), since on expanding out the left-hand side of \eqref{gu-1}, one obtains an alternating sum of $2^{2^{D_*+2D}}$ terms, all of which are $1+o(1)$ thanks to Theorem \ref{poly} (with $s = 2r + D_*+2D$, and $k$ at most $2^{D_*+2D}$).  In both of these cases, a direct inspection reveals that the polynomials $P_i$ used in the invocation of Theorem \ref{poly} have non-constant differences $P_i-P_j$.

In a similar vein, Lemma \ref{gopw} follows from Theorem \ref{poly} (now with the $b_i$ all distinct), as the left-hand side of \eqref{gopw} expands as an expression of the form considered by Theorem \ref{poly} (with $s = D_{\mathcal S} + r + 2r (|{\mathcal A}_l|-1)$ and $k = 2^{|{\mathcal A}_l|-1}$).  Similarly, for Lemma \ref{st-lemma}, the left-hand side of \eqref{st1} expands as $2^{|{\mathcal A}_0|}$ terms, all of which are $1+o(1)$ by Theorem \ref{poly} (with $s = D_{\mathcal S}$ and $k$ at most $|{\mathcal A}_0|$), and the left-hand side of \eqref{st2} similarly expands as the sum of $2^{|{\mathcal A}_0|+2|{\mathcal A}_1|}$ terms, which are again $1+o(1)$ by Theorem \ref{poly} (with $s = D_{\mathcal S} + 2r$ and $k$ at most $|{\mathcal A}_0| + 2 |{\mathcal A}_1|$); it is in this latter case that we need to permit the non-constant coefficients of the polynomials $P_i$ in Theorem \ref{poly} to be larger than $A^C$ in magnitude.  As before, an inspection of the polynomials involved (using the fact that the $R_\alpha-R_\beta$ are non-constant) shows that the $P_i-P_j$ are non-constant.

Finally, the expression \eqref{task} is $1+o(1)$ by an application of Theorem \ref{poly} with $s = D_{\mathcal S} + 2r |{\mathcal A}_l| + 2r$ and $k = |{\mathcal A}_{nl}| 2^{|{\mathcal A}_l|} + |{\mathcal A}_l| 2^{|{\mathcal A}_l|-1}$.  (Again, by focusing on the behavior with respect to the $\vec m^{(i)}$ variables, setting all other variables to zero, one can use the hypotheses on the $R_\alpha -R_\beta$ to show that the polynomials $P_i-P_j$ are non-constant.)

\end{document}